\documentclass[12pt]{article}

\usepackage{amstext    }
\usepackage{amsthm    }
\usepackage{a4}
\usepackage[mathscr]{eucal}
\usepackage{mathrsfs}

\usepackage{amsmath}
\usepackage{amssymb}
\usepackage{amscd}

\newtheorem{theorem}{Theorem}[section]

\newtheorem{proposition}[theorem]{Proposition}
\newtheorem{corollary}[theorem]{Corollary}
\newtheorem{lemma}[theorem]{Lemma}
\newtheorem{remark}[theorem]{Remark}

\newcommand{\cali}[1]{\mathscr{#1}}

\newcommand{\supp}{{\rm supp}}

\newcommand{\dist}{{\rm dist}}

\newcommand{\ddc}{{dd^c}}
\newcommand{\dc}{{d^c}}
\newcommand{\dbar}{{\overline\partial}}
\newcommand{\ddbar}{{\partial\overline\partial}}

\newcommand{\Dom}{{\rm Dom}}

\renewcommand{\div}{{\rm div}}
\renewcommand{\Re}{{\rm Re}}

\newcommand{\Cc}{\cali{C}}
\newcommand{\Dc}{\cali{D}}
\newcommand{\Ec}{\cali{E}}
\newcommand{\Fc}{\cali{F}}
\newcommand{\Gc}{\cali{G}}
\newcommand{\Hc}{\cali{H}}

\newcommand{\Lc}{\cali{L}}

\newcommand{\C}{\mathbb{C}}
\newcommand{\D}{\mathbb{D}}

\newcommand{\N}{\mathbb{N}}

\newcommand{\R}{\mathbb{R}}
\newcommand{\T}{\mathbb{T}}
\newcommand{\B}{\mathbb{B}}
\newcommand{\U}{\mathbb{U}}
\renewcommand{\S}{\mathbb{S}}
\renewcommand{\P}{\mathbb{P}}

\title{Heat equation and ergodic theorems for Riemann surface
  laminations}

\author{T.-C. Dinh, V.-A. Nguy{\^e}n and N. Sibony}

\begin{document}

\maketitle

\begin{abstract}
We introduce the heat equation relative to a positive $\ddbar$-closed current
and apply it to the invariant currents associated with Riemann surface
laminations possibly with singularities. The main examples are holomorphic foliations by Riemann surfaces in  projective spaces. We prove two kinds of
ergodic theorems for such currents: one associated to the heat
diffusion and one close to Birkhoff's averaging on orbits of a
dynamical system. The heat diffusion theorem with respect to a harmonic measure 
is also developed for real laminations.
\end{abstract}

\noindent
{\bf Classification AMS 2010:} 37F75, 37A.


\bigskip
\noindent
{\bf Notation.} Throughout the paper, $\D$ denotes the unit disc
in $\C$, $r\D$ denotes the disc of center 0 and of radius $r$ and
$\D_R\subset\D$ is the disc of center 0 and of radius $R$ with
respect to the Poincar{\'e} metric on $\D$,
i.e. $\D_R=r\D$ with $R:=\log[(1+r)/(1-r)]$. Poincar{\'e}'s metric on
a Riemann surface, in particular on $\D$ and on the leaves of a lamination, is
given by a positive $(1,1)$-form that we denote by $\omega_P$. 
The notation $\U\simeq \B\times\T$ is a flow box which is often
identified with an open set of the lamination. Here, $\T$ is a transversal
and $\B$ is an open set in $\R^n$ for real laminations or in $\C$ for
Riemann surface laminations.

\section{Introduction} \label{intro}

When  $\Fc$ is  a smooth foliation  of a compact Riemannian manifold $(M,g)$  with  
smooth leaves (or more generally a lamination by Riemannian
manifolds), L.  Garnett  \cite{Garnett}  has 
studied  a diffusion process on the leaves of the foliation. 
The metric  $g$ induces
a Laplace operator $\Delta$ along the leaves and the diffusion  process  is  associated to  the  heat  equation
$$\frac{du}{dt}-\Delta u=0\quad \mbox{and} \quad 
u(0,\cdot)=u_0.$$
Since  the  leaves  have  bounded  geometry, the  classical theory, based on Malliavin \cite{Malliavin} and McKean \cite{McKean} analytic estimates,
applies and  
one can study the  diffusion process  associated to that  equation.
L.  Garnett  proved also  an ergodic theorem  for the semi-group
$S(t)$  of diffusion operators
associated to the heat equation.  Recall that a 
positive harmonic measure for $\Fc$ is  a positive  measure $m$
such that   $ \langle \Delta u,m\rangle=0$
for all smooth  functions $u$.  It can be decomposed in a flow box 
as an average of measures on plaques which are given by harmonic
forms. 
We refer to Candel-Conlon \cite{Candel2,CandelConlon2}   
and Walczak \cite{Walczak} for a more recent treatment of L. Garnett's
theory. 
Their approach relies on uniform estimates of the heat kernel 
using that the leaves have bounded geometry.

The theory does not apply to Riemann surface laminations with singularities,
e.g. to the study of polynomial  vector fields in $\C^{k+1},$  
for which we can associate a foliation in
$\P^k$.  More  precisely,  let  
$$F(z):=\sum_{j=0}^k  F_j(z) \frac{\partial}{\partial z_j} $$
with $F_j$ homogeneous polynomials of degree $d \geq 1.$  It induces  a  
foliation with singularities  in $\P^k.$  The  singularities correspond either to
indeterminacy points of  $F= [F_0  : \cdots : F_k ]$ or to fixed
points of $F$ in $\P^k$.  
In general, the leaves  are  not of bounded  geometry nor even complete for the induced metric. They have bounded geometry with respect to the Poincar{\'e} metric on the leaves that we will consider, but then the metric is not in general transversally continuous. For general results on foliations in $\P^k$, see the book by Ilyashenko-Yakovenko \cite{IlyashenkoYakovenko} which focuses on dimension 2. The survey by Forn\ae ss-Sibony \cite{FornaessSibony2} emphasizes the use of currents.

In this paper we  construct  the heat diffusion in a slightly
different context.  
We  consider laminations by Riemann surfaces with singularities
in a compact hermitian   manifold or abstract laminations by
Riemannian leaves without singularities.   
In the real case, assume that we have a  Laplacian  $\Delta$ along leaves  such that
for a test function $u$,  regular enough,     
$\Delta u$  is  continuous. Then    an application of the  
Hahn-Banach theorem permits to obtain a
harmonic measure $m,$ 
see e.g. Garnett \cite{Garnett}. 
In the complex case with singularities, the construction of $m$ is
different, see \cite{BerndtssonSibony}. One has to use the notion of
plurisubharmonic functions which have the property to be subharmonic
on every leaf independently of the lamination.  One can also use an averaging process as in \cite{FornaessSibony1}.  

When  a  harmonic  measure is  given  we  can
consider 
other natural Laplacians, which vary only measurably and $m$  is still harmonic with respect to
these  Laplacians.  We then introduce  a heat equation associated to $m.$
We can develop a  Hilbert space  theory  with respect to   that
equation 
using Lax-Milgram  and  Hille-Yosida  theorems.  More precisely, given
$u_0$ in the domain $\Dom(\Delta)$ of $\Delta$,
we solve 
$$\frac{du}{dt}=\Delta u\quad \mbox{and}\quad
  u(0,\cdot)=u_0$$
with $u(t,\cdot)\in \Dom(\Delta)$. 
The theory is  sufficient to get an ergodic theorem for  that diffusion. So, we rather get the heat equation in the space $(M,\Fc,m)$. The Laplacians are not necessarily symmetric operators in $L^2(m)$ and the natural ones depend on $m$.

We give a self-contained proof of the ergodic theorem in the Riemannian case without any use of bounded geometry nor delicate estimates on the heat kernel. This will permit further generalizations. 
We get also the mixing for the diffusion associated to a natural Laplacian with coefficients defined only $m$-almost everywhere.
We then apply the same ideas to the complex case with singularities. In some sense, we treat harmonic measures and $\ddc$-closed currents as manifolds and we solve the heat equation relatively to those measures and currents. The case of $\dbar$-equation induced on a current was studied by Berndtsson and the third author in \cite{BerndtssonSibony}. 

In the  second part, for compact Riemann  surface laminations  with
singularities   
we get   an ergodic theorem  with more geometric  flavor than the ones  associated  to
a diffusion.  Let $(X,\Lc,E)$  be a lamination
by Riemann  surfaces.   Assume for simplicity that the  singularity set $E$ of $\Lc$ is
a  finite  set of points (several results still hold for a tame singular set, e.g. a complete pluripolar set).  
Then  every hyperbolic leaf $L$ is  covered by the unit disc $\D.$
Let $\phi_a:\D\rightarrow L_a$  denote  a universal
covering map of  the leaf $L_a$ passing through $a$ with  $\phi_a(0)=a.$
We  consider the  associated measure
$$m_{a,R}:=\frac{1}{M_R}(\phi_a)_* \big( \log^+ \frac{r}{|
  \zeta|}\omega_P\big) \quad\mbox{with}\quad 
R:=\log{1+r\over 1-r}$$
which is obtained by averaging until ``hyperbolic time'' $R$ along the leaves. Here,
$M_R$ is  a constant to normalize the mass.
Recall that $\omega_P$ denotes the Poincar{\'e} metric on $\D$ and also on
the leaves of $X$. 

Let $T$ be an extremal positive harmonic current on $X$ directed by the lamination.
Consider the measure $m_P:=T\wedge \omega_P$ (which is always finite when the singularities are linearizable) and for simplicity assume that $m_P$ is a probability measure. 
So, this is a natural harmonic measure on $X$ having no mass on parabolic leaves. We prove in particular that $m_{a,R}$ 
tends to $m_P$ for 
$m_P$-almost every $a$. This is  a lamination  version 
of the classical Birkhoff Theorem. 
Here, we introduce operators $B_R$, $R\in\R_+$, 
which are the analogue, for hyperbolic foliations, 
of the Birkhoff sums in discrete dynamics, see also 
Bonatti-G{\'o}mez-Mont-Viana \cite{BonattiGomezMont,BGMV} and \cite{FornaessSibony1}. 
For a test function $u$, the function $B_Ru$ is given by 
$$B_Ru(a):=\langle m_{a,R},u\rangle.$$
Our result is equivalent to the convergence $B_Ru(a)\to \langle
m_P,u\rangle$ for $m_P$-almost every $a$ and for $u\in L^1(m_P)$.

\section{Currents on a lamination} \label{section_current}

In this section, we will give some basic notions and properties of
currents for laminations. We refer to Demailly \cite{Demailly} and
Federer \cite{Federer} for currents on manifolds.
 
Let $X$ be a locally compact space. Consider an atlas $\Lc$ of $X$ with charts 
$$\Phi_i:\U_i\rightarrow \B_i\times \T_i,$$
where $\T_i$ is a locally compact metric space, $\B_i$ is a domain in $\R^n$ and
$\Phi_i$ is a homeomorphism defined on 
an open subset $\U_i$ of
$X$. 
We say that $(X,\Lc)$ is a
{\it real lamination of dimension $n$} if all the 
changes of coordinates $\Phi_i\circ\Phi_j^{-1}$ are of the form
$$(x,t)\mapsto (x',t'), \quad x'=\Psi(x,t),\quad t'=\Lambda(t)$$
where $\Psi,\Lambda$ are continuous functions, $\Psi$ is smooth with
respect to $x$ and its partial derivatives of any order with respect to $x$ are continuous.

The open set $\U_i$ is called a {\it flow
  box} and the manifold $\Phi_i^{-1}\{t=c\}$ in $\U_i$ with $c\in\T_i$ is a {\it
  plaque}. The property of the above coordinate changes insures that
the plaques in different flow boxes are compatible in the intersection of
the boxes.
A {\it leaf} $L$ is a minimal connected subset of $X$ such
that if $L$ intersects a plaque, it contains the plaque. So, a leaf $L$
is a connected real manifold of dimension $n$ immersed in $X$ which is a
union of plaques. It is not difficult to see that $\overline L$ is
also a lamination.
We will only consider {\it oriented laminations}, i.e. the case where
$\Phi_i$ preserve the canonical orientation on $\R^n$. So, the leaves
of $X$ inherit the orientation given by the one of $\R^n$.
A {\it transversal} in a flow box is a closed set of the box which intersects every
plaque in one point. In particular, $\Phi_i^{-1}(\{x\}\times \T_i)$ is a
transversal in $\U_i$ for any $x\in \B_i$.  In order to
simplify the notation, we often identify $\T_i$
with $\Phi_i^{-1}(\{x\}\times \T)$ for some $x\in \B_i$ or even
identify $\U_i$ with $\B_i\times\T_i$ via the map $\Phi_i$.

From now on, we fix an atlas on $X$ which is locally finite. For
simplicity, assume that the associated local coordinates extend to a
neighbourhood of the closure of each flow box in the atlas. If $\Phi:\U\rightarrow
\B\times \T$ is such a flow box, we assume for simplicity that $\B$
is contained in the ball of center 0 and of radius 3 in $\R^n$ and $\T$ is a locally compact 
metric space of diameter $\leq 1$. If the lamination is embedded in a Riemannian
manifold, it is natural to consider the metric on that manifold. In
the abstract setting, it is useful to introduce a metric for the
lamination. First, consider the metric on the flow box $\U\simeq
\B\times \T$ which is induced by the ones
on $\R^n$ and on $\T$. So, the flow box has diameter $\leq
7$ with respect to this metric. 
Consider two points $a,b\in X$, a sequence $a_0,\ldots,a_m$ with
$a_0=a$, $a_m=b$ and $a_i,a_{i+1}$ in a same flow box $\U_i$. Denote by
$l_i$ the distance between $a_i$ and $a_{i+1}$ in $\U_i$. 
Define the
distance between $a,b$ as the infimum of $\sum l_i$ over all choices
of $a_i$ and $\U_i$. This distance is locally equivalent to the
distance in flow boxes.

We recall now the notion of currents on a manifold.
Let $M$ be a real oriented manifold of dimension $n$. We fix an atlas
of $M$ which is locally finite. Up to reducing slightly the charts, we
can assume that the local coordinates system associated to
each chart is defined on a neighbourhood of the closure of this chart. For $0\leq p
\leq n$ and $l\in\N$, denote by $\Dc^p_l(M)$ the space of $p$-forms of
class $\Cc^l$ with compact support in $X$ and $\Dc^p(X)$ their
intersection. If $\alpha$ is a $p$-form on $X$, denote by
$\|\alpha\|_{\Cc^l}$ the sum of the $\Cc^l$-norms of the coefficients
of $\alpha$ in the local coordinates. These norms induce a
topology on $\Dc^p_l(M)$ and $\Dc^p(M)$. In particular, a
sequence $\alpha_j$ converges to $\alpha$ in $\Dc^p(M)$ if
these forms are supported in a fixed compact set and if
$\|\alpha_j-\alpha\|_{\Cc^l}\rightarrow 0$ for every $l$. 

A current of degree $p$ and of dimension $n-p$ on $M$ (a $p$-current
for short) is a continuous linear form
$T$ on $\Dc^{n-p}(M)$ with values in $\C$. 
The value of $T$ on a test form $\alpha$ in  $\Dc^{n-p}(M)$ is
denoted by $\langle T,\alpha\rangle$ or $T(\alpha)$.
The current $T$ is
of order $\leq l$ if it can be extended to a continuous linear form on 
$\Dc^{n-p}_l(M)$. The order of $T$ is the minimal integer $l\geq 0$ satisfying this condition.
It is not difficult to see that the restriction of $T$ to a relatively compact open
set of $M$ is always of finite order. Define
$$\|T\|_{-l,K}:=\sup\Big\{|\langle T,\alpha\rangle|,\quad
\alpha\in\Dc^{n-p}(M),\quad \|\alpha\|_{\Cc^l}\leq 1,\quad
\supp(\alpha)\subset K\Big\}$$
for $l\in\N$ and $K$ a compact subset of $M$. 
This quantity may be infinite when the order of $T$ is larger than $l$.

Consider now a real lamination of dimension $n$ as above.
The notion of differential forms on manifolds can be extended
to laminations, see Sullivan \cite{Sullivan}. A $p$-form on $X$ can be seen on the flow box
$\U\simeq \B\times\T$ as
a $p$-form on $\B$ depending on the parameter $t\in \T$. 
For $0\leq p\leq n$, denote by $\Dc^p_l(X)$ the space of 
$p$-forms $\alpha$ with compact support satisfying the following property: 
$\alpha$ restricted to each flow box $\U\simeq \B\times\T$ is a 
$p$-form of class $\Cc^l$ on the plaques whose coefficients and
all their derivatives up to order $l$ depend continuously on the
plaque. The norm $\|\cdot\|_{\Cc^l}$ on
this space is defined as in the case of real manifold using a locally finite
atlas of $X$. We also define $\Dc^p(X)$ as the intersection of
$\Dc^p_l(X)$ for $l\geq 0$.
{\it A current of bidegree $p$ and of dimension $n-p$} on $X$ is a continuous
linear form on the space $\Dc^{n-p}(X)$ with values in $\C$. A
$p$-current is of order $\leq l$ if it can be extended to a linear
continuous form on $\Dc^{n-p}_l(X)$. The restriction of a current to a
relatively compact open set of $X$ is always of finite order. 
The norm $\|\cdot\|_{-l,K}$ on currents is defined as in the case of
manifolds. The following result gives us the local structure of a
current. It shows in particular that we can consider the restriction
of a current to a measurable family of plaques.

\begin{proposition} \label{prop_struc_current}
Let $T$ be a $p$-current on a lamination $X$ and let
$\U\simeq \B\times\T$ be
a flow box as above which is relatively compact in $X$. 
Let $l$ be the order of the restriction of $T$ to $\U$. 
Then there is a positive Radon measure $\mu$ on $\T$ and a
measurable family of $p$-currents $T_a$ of order $l$ on $\B$ for 
$\mu$-almost every $a\in \T$ such that if $K$ is compact in $\B$, the integral
$\langle\mu, \|T_a\|_{-l,K}\rangle$ is finite and 
$$\langle T,\alpha\rangle =\int_\T \langle T_a,\alpha(\cdot,a)\rangle
d\mu(a)\quad
\mbox{for } \alpha \in \Dc^{n-p}_l(\U).$$ 
\end{proposition}
\proof
Observe that if we have a local disintegration as above for two currents $T,T'$, then it is easy to get such a disintegration for $T+T'$ using the sum $\mu+\mu'$ of the corresponding measures $\mu,\mu'$. This property allows us to make several reductions below.
Using a partition of unity, we can reduce the problem to the case
where $T$ has compact support in $\U\subset \R^n\times\T$. If $(x_1,\ldots,x_n)$
  is a coordinate system in $\R^n$ and $1\leq i_1<\cdots<i_{n-p}\leq n$,
we only have to prove the proposition for the current $T\wedge dx_{i_1}\wedge\ldots \wedge
dx_{i_{n-p}}$. Therefore, we can 
assume that $T$ is an $n$-current, i.e. a distribution.

Now, since $T$ is of order $l$, it can be seen as a continuous linear
form on the derivatives $\varphi_I$ of order $l$ of a test function
$\varphi\in\Dc_l^0(X)$. By Hahn-Banach
theorem, there are distributions $T_I$ of order $0$ such that $T=\sum
T_I(\varphi_I)$. It is enough to prove the proposition for each $T_I$
instead of $T$. So, we can assume that
$T$ is of order 0, i.e.  a Radon measure. 

Since $T$ can be written as a difference of two positive measures, we
only have  to consider the case where $T$ is positive. Define 
$\pi$ the canonical projection from $U$ to $\T$ and
$\mu:=\pi_*(T)$.
The disintegration of $T$ along the fibers of $\pi$ gives the result.
\endproof

The following result shows that the above local decomposition of a
current is almost unique.

\begin{proposition} \label{prop_struc_current_unique}
With the notation of Proposition \ref{prop_struc_current}, if $\mu'$
and $T'_a$ are associated with another decomposition of $T$ in $\U$, then
there is a measurable function $\lambda>0$ on a measurable set
$\S\subset \T$ such that $T_a=0$ for $\mu$-almost every $a\not\in \S$,
 $T_a'=0$ for $\mu'$-almost every $a\not\in \S$,
$\mu'=\lambda\mu$ on $\S$
and $T_a=\lambda(a)T_a'$ for $\mu$ and
$\mu'$-almost every $a\in\S$.
\end{proposition}
\proof
Consider first the case where $T=0$. We show that $T_a=0$ for
$\mu$-almost every $a$. Let $\alpha$ be a test form in
$\Dc^{n-p}(\B)$. Define $\eta(a):=\langle T_a,\alpha\rangle$. If $\chi$
is a continuous function with compact support in $\T$, we have by
Proposition \ref{prop_struc_current}
$$\int\chi\eta d\mu= \langle T,\chi\alpha\rangle=0.$$
It follows that $\eta(a)=0$ for $\mu$-almost every $a$. Applying this
property to a dense sequence of test forms $\alpha_j\in\Dc^{n-p}(\B)$
allows us to conclude that $T_a=0$ for $\mu$-almost every $a$.

Consider now the general case. Define
$$\S:=\{a\in\T,\ T_a\not=0, T_a'\not=0\}.$$
Observe that the restriction of $T$ to 
$$E:=\{a\in\T,\ T_a\not=0,\ T_a'=0\} \quad \mbox{and}\quad 
E':=\{a\in\T,\ T_a=0,\ T_a'\not=0\}$$ 
vanishes.  Then, using the first case, we obtain
that $\mu$ has no mass on $E$ and $\mu'$ has no mass on $E'$.
Therefore, $T_a=0$ for $\mu$-almost every $a\not\in\S$ and 
$T_a'=0$ for $\mu'$-almost every $a\not\in\S$.
Using the
first case, we also deduce that a measurable subset of $\S$ has
positive $\mu$ measure if and only if this is the case for
$\mu'$. Therefore, there is a function $\lambda>0$ such that
$\mu'=\lambda\mu$ on $\S$. Observe that $T_a-\lambda T_a'$ and $\mu$ define a
decomposition of 0 in $\U$. It follows that $T_a=\lambda T_a'$ for
$\mu$-almost every $a\in\S$.
\endproof 

A 0-current $T$ on a lamination is {\it positive} if it is of order 0 and if in
the local description as in Proposition \ref{prop_struc_current}, the currents
$T_a$ are given by  positive functions on $\B$ for $\mu$-almost every
$a\in \T$. Consider a Riemannian metric $g$ which is smooth on the leaves of $X$ and such that its restriction to a flow box depends in a measurable way on the plaques. 
A 0-current $T$ of order 0 is called {\it $g$-harmonic} (or simply {\it harmonic} if there is no confusion) if in the local description as above, the currents $T_a$ are given by $g$-harmonic functions.

Consider now the complex setting.
In the definition of the lamination $(X,\Lc)$, 
when the $\B_i$ are domains in $\C^n$ and $\Phi_i$ are holomorphic with
respect to $x$, we say that $X$ is a {\it complex lamination of dimension
  $n$}. In this case, the complex structure on $\B_i$ induces a complex
structure on the leaves of $X$. Therefore,
in the definition of lamination, it is enough to
assume that $\Psi_i(x,t)$
depends continuously on $t$; indeed Cauchy's  formula implies that all partial
derivatives of this function with respect to $x$ satisfy the same
property. 

Let $X$ be a complex lamination of dimension $n$. 
Denote by $\Dc^{p,q}_l(X)$ and $\Dc^{p,q}(X)$ the spaces of
forms in $\Dc^{p+q}_l(X)$ and $\Dc^{p+q}(X)$ respectively whose
restriction to plaques is of bidegree $(p,q)$. A $(p+q)$-current is of
bidegree $(p,q)$ if it vanishes on forms of bidegree $(n-p',n-q')$
with $(p',q')\not=(p,q)$.
The operators $\partial$ and $\overline\partial$ act on currents as in
the case of manifolds. If
$T$ is a $(p,q)$-current then $\partial T$ and $\dbar T$
are defined by
$$\langle\partial T,\alpha\rangle:=-\langle T,\partial \alpha\rangle
\quad \mbox{for all test } (n-p-1,n-q)\mbox{-form } \alpha$$
and
$$\langle\dbar T,\alpha\rangle:=-\langle T,\dbar \alpha\rangle
\quad \mbox{for all test } (n-p,n-q-1)\mbox{-form } \alpha.$$
We call {\it $\ddbar$-closed current or pluriharmonic current} a 
$(0,0)$-current $T$ on $X$ such that
$\ddbar T=0$ (in dimension $n=1$, we will say simply ``harmonic"
instead of ``pluriharmonic"). The following result shows that 
in dimension $n=1$ this notion coincides with the notion of $g$-harmonic current if we consider 
Poincar{\'e}'s metric on the hyperbolic leaves of $X$ and standard metrics
on the parabolic ones or more generally conformal metrics on the leaves.

\begin{proposition} \label{prop_current_local_c}
Let $T$ be a pluriharmonic current on a complex lamination $X$. Let $\U\simeq \B\times\T$ be a flow
box as above which is relatively compact in $X$. Then $T$ is a normal current, i.e. $T$ and $dT$ are 
of order $0$. Moreover, there is a positive Radon
measure $\mu$ on $\T$ and for $\mu$-almost every $a\in \T$ there is a
pluriharmonic function $h_a$ on $\B$ 
such that if $K$ is compact in $\B$ 
the integral $\langle \mu, \|h_a\|_{L^1(K)}\rangle$ is finite and
$$\langle T,\alpha\rangle=\int_\T \Big(\int_\B h_a(z) \alpha(z,a)\Big) d\mu(a)
\quad \mbox{for}\quad  \alpha\in \Dc^{n,n}_0(X).$$
\end{proposition}
\proof
Using Proposition \ref{prop_struc_current}, we easily deduce that $T_a$ is
$\ddbar$-closed for $\mu$-almost every $a\in \T$. It follows that
$T_a$ is given by a pluriharmonic function $h_a$ on $\B$. If $K,L$ are  compacts in $\B$ with $K\Subset L$, 
the harmonic property implies that $\|h_a\|_{L^1(K)}\lesssim \|h_a\|_{-l,L}$ for $0\leq l<\infty$ and 
$\|dh_a\|_{L^1(K)}\lesssim \|h_a\|_{L^1(L)}$. 
This implies that $T$ and $dT$ are of order 0 and completes the proof.
\endproof

For complex lamination, there is a notion of positivity for
currents of bidegree $(p,p)$ which extends the same notion for
$(p,p)$-currents on complex manifolds. We shortly recall the last
one that we will use later.

A $(p,p)$-form on a complex manifold $M$ of dimension $n$ is {\it
  positive} if it can be written at every point as a combination with
positive coefficients of forms of type
$$i\alpha_1\wedge\overline\alpha_1\wedge\ldots\wedge
i\alpha_p\wedge\overline\alpha_p$$
where the $\alpha_j$ are $(1,0)$-forms. A $(p,p)$-current or a $(p,p)$-form $T$ on $M$ is
{\it weakly positive} if $T\wedge\varphi$ is a positive measure for
any smooth positive $(n-p,n-p)$-form $\varphi$. A $(p,p)$-current $T$
is {\it positive} if $T\wedge\varphi$ is a positive measure for
any smooth weakly positive $(n-p,n-p)$-form $\varphi$. 
If $M$ is given with a Hermitian metric $\omega$, 
$T\wedge \omega^{n-p}$ is a positive measure on
$M$. The mass of  $T\wedge \omega^{n-p}$
on a measurable set $E$ is denoted by $\|T\|_E$ and is called {\it the mass of $T$ on $E$}.
{\it The mass} $\|T\|$ of $T$ is the total mass of  $T\wedge \omega^{n-p}$.
We will use the following local property of positive $\ddbar$-closed
currents which is due to Skoda \cite{Skoda}.
Recall that $\dc:={i\over2\pi}(\dbar-\partial)$ and $\ddc = {i\over \pi}\ddbar$. 
\begin{lemma} \label{lemma_skoda}
Let $B_r$ denote the ball of
  center $0$ and of radius $r$ in $\C^n$. 
Let $T$ be a positive $\ddbar$-closed $(p,p)$-current in
  a ball $B_{r_0}$.
Define $\beta:=\ddc\|z\|^2$
  the standard K{\"a}hler form where $z$ is the canonical
  coordinates on $\C^n$.
Then the function $r\mapsto \pi^{-(n-p)} r^{-2(n-p)}\|T\wedge\beta^{n-p}\|_{B_r}$
is increasing on $0<r\leq r_0$. In particular, it is bounded on 
$]0,r_1]$ for any  $0<r_1<r_0$.
\end{lemma}

The limit of the above function when $r\rightarrow 0$ is called {\it
  the Lelong number} of $T$ at $0$. The lemma shows that Lelong's
number exists and is finite.

\begin{lemma} \label{lemma_ds}
Let $T$ be a positive current of bidimension $(1,1)$ with compact support on a complex manifold $M$. Assume that $\ddc T$ is a negative measure on $M\setminus E$ where $E$ is a finite set. Then $T$ is a $\ddc$-closed current on $M$.
\end{lemma}
\proof
Since $E$ is finite, $\ddc T$ is a negative measure on $M$, see \cite{FSW}. On the other hand, we have 
$$\langle \ddc T,1\rangle=\langle T,\ddc 1\rangle =0.$$
It follows that $\ddc T=0$ on $M$. 
\endproof

\section{Riemann surface laminations} \label{section_riemann}

In this section, we consider a {\it Riemann surface lamination}, i.e. a
complex lamination $X$ as above of dimension $n=1$. 
The lamination has no singular points but we do not assume that it is compact.
What we have in mind as an example is the regular part of a compact lamination with singularities. 
Consider also a Hermitian metric on $X$, i.e. Hermitian metrics on the
leaves of $L$ whose restriction to each flow box defines Hermitian
metrics on the plaques that depend continuously on the plaques. It is
not difficult to construct such a metric using a partition of
unity. Observe that all the Hermitian metrics on $X$ are locally
equivalent. So, from now on, fix a Hermitian metric on $X$. It is
given by a strictly positive smooth $(1,1)$-form $\omega$ on $X$. 

We will need some basic properties of these laminations. Let $S$ be a
hyperbolic Riemann surface, i.e. a Riemann surface whose universal
covering is the unit disc $\D$ in $\C$. Let $\phi:\D\rightarrow
S$ be a universal covering map which is unique up to an automorphism on $\D$. 
The fundamental group $\pi_1(S)$ can be identified with a group of
automorphisms of $\D$.
Since the Poincar{\'e} metric on
$\D$ is invariant under the automorphism group, it
induces via $\phi$ a metric on $S$ that we also call {\it the Poincar{\'e}
  metric}. It is smooth and the surface $S$ is complete with respect to that metric.
By convention, Poincar{\'e}'s metric (pseudo-metric to be precise) on a
parabolic Riemann surface vanishes identically. 

Poincar{\'e}'s metric on the leaves of $X$ defines a positive
$(1,1)$-form $\omega_P$, which a priori is not necessarily transversally continuous. 
The continuity is proved in some important cases, see Candel-G{\'o}mez-Mont \cite{CandelGomezMont} and \cite{FornaessSibony2}. 
Consider a hyperbolic leaf $L_a$ passing through
a point $a$ and a universal covering map $\phi_a:\D\rightarrow L_a$ 
such that $\phi_a(0)=a$. The map $\phi_a$ is unique up to a rotation on $\D$. 
Define 
$$\vartheta(a):=\|D\phi_a(0)\|^{-2},$$
where $\|D\phi_a(0)\|$ is the norm of the differential of $\pi$ at $0$
with respect to the Euclidian metric on $\D$ and the fixed
Hermitian metric on $L$. Recall that at 0 the Poincar{\'e} metric on $\D$ 
is equal to two times the Euclidian metric. The above definition does not depend on the
choice of $\phi_a$.
We obtain that 
$$\omega_P=4\vartheta\omega.$$
Recall also that $\omega_P$ is an extremal metric in the sense that 
if $\tau:\D\rightarrow L_a$ is a holomorphic map
such that $\tau(0)=a$, then $\|D\tau(0)\|\leq \vartheta(a)^{-1/2}$. The
equality occurs in the last estimate only when $\tau$ is a universal
covering map of $L_a$. 

Consider an open set $V\subset \C$ and a sequence of holomorphic maps
$\tau_n: V\rightarrow X$, 
i.e. holomorphic maps from $V$ to leaves of $X$. We say that
$\tau_n$ {\it converge locally uniformly} to a holomorphic map
$\tau:V\rightarrow X$ if any point $z_0\in V$ admits a
neighbourhood $V_0$ such that for $n$ large enough $\tau_n$ and $\tau$ restricted
to $V_0$ have values in the same flow box and $\tau_n$ converge
uniformly to $\tau$ on $V_0$. This notion coincides with the
local uniform convergence with respect to the metric on $X$ introduced
in Section \ref{section_current}.
A family $\Fc$ of holomorphic maps from
$V$ to $X$ is said to be {\it normal} if any infinite set
$\Fc'\subset \Fc$ admits a sequence which converges locally uniformly
to a holomorphic map. We have the following proposition.

\begin{proposition} \label{prop_poincare_measurable}
Let $X$ be a Riemann surface lamination as above. Then the
  Poincar{\'e} metric $\omega_P$ is a measurable $(1,1)$-form on
  $X$. In particular, the union of parabolic leaves is a measurable
  set. Moreover, the function $\vartheta$ associated with $\omega_P$
  is locally bounded.
\end{proposition}
\proof
By definition, $\vartheta$ is a non-negative function. We first show that it is
locally bounded. Consider a small neighbourhood $W$ of a point $a\in
X$ and a flow box $\U\simeq \B\times \T$ containing $W$. We can assume
that $\B$ is the disc of center 0 and of radius 3 in $\C$ and that $W$
is contained in $\D\times\T$ where $\D$ is the unit disc in
$\C$. Consider the family of holomorphic map $\tau:\D\rightarrow
\U$ such that $\tau(z)=(z+b,t)$ with $(b,t)\in W$. It is clear that
$\|D\tau(0)\|$ is bounded from below by a strictly positive constant
independent of $(b,t)$. Therefore, the extremality of Poincar{\'e}'s
metric implies that $\vartheta$ is bounded from above on $W$. This gives
the last assertion in the proposition. 

It remains to show that $\vartheta$ is a measurable function. 
Fix a sequence $K_n$ of compact subsets of $X$
such that $K_n$ is contained in the interior of $K_{n+1}$ and that
$X=\cup K_n$. 
We only have to show that $\vartheta$ is measurable on
$K_0$.
For all positive integer $n$, denote by
$\Fc_n$ the family of holomorphic maps $\tau:\D\rightarrow
K_n$ such that $\|D\tau\|_\infty\leq n$. 
It is not difficult to see
using flox boxes that this family is compact. Therefore, the function
$$\xi_n(a):=\sup\Big\{\|D\tau(0)\|:\quad \tau\in
\Fc_n\mbox{ with } \tau(0)=a \Big\}$$
is upper semi-continuous on $a\in K_0$. The extremality of Poincar{\'e}'s metric
implies that $\xi_n\leq \vartheta^{-1/2}$. It is now enough to show that
$\vartheta^{-1/2}=\sup_n \xi_n$. We distinguish two cases.

Let $a\in K_0$ be a point such that $L_a$ is hyperbolic and consider a universal covering map
$\phi_a:\D\rightarrow L_a$ with $\phi_a(0)=a$.  
Define $\tau_r:=\phi_a(rz)$ for $0<r<1$. It is clear
that $\tau_r(\D)$ is relatively compact in $X$ and $\|D\tau_r\|$
is bounded. So, $\tau_r$ belongs to $\Fc_n$ for $n$ large
enough. On the other hand, we have
$$\vartheta^{-1/2}(a)=\|D\phi_a(0)\|=\lim_{r\rightarrow 1} \|D\tau_r(0)\|.$$
Therefore, $\vartheta^{-1/2}=\sup_n \xi_n$ on hyperbolic leaves. 

Let $a\in K_0$ be a point such that $L_a$ is parabolic and consider a map
$\phi_a:\C\rightarrow L_a$ with $\phi_a(0)=a$. Define
$\tau_r(z):=\phi_a(rz)$. It is also clear that $\tau_r$ belongs to $\Fc_n$ for $n$ large
enough and we have
$$\vartheta^{-1/2}(a)=+\infty=\lim_{r\rightarrow \infty } \|D\tau_r(0)\|.$$
It follows that $\vartheta^{-1/2}=\sup_{n\geq 0} \xi_n$
and this completes the proof.
\endproof

Consider now a flow box $\Phi:\U\rightarrow \B\times \T$ as
above. Recall that for simplicity,  we identify $\U$ with $\B\times\T$
and $\T$ with the transversal
$\Phi^{-1}(\{z\}\times\T)$ for some point $z\in \B$.
We have the following result.

\begin{proposition} \label{prop_lusine}
Let  $\nu$ be a positive Radon 
  measure on $\T$. Let
  $\T_1\subset \T$ be a measurable set such that $\nu(\T_1)>0$ and $L_a$ is
  hyperbolic for any $a\in\T_1$. Then for every $\epsilon>0$ there
  is a compact set $\T_2\subset \T_1$ with $\nu(\T_2)>\nu(\T_1)-\epsilon$ 
and a family of universal covering maps $\phi_a:\D\rightarrow L_a$
with $\phi_a(0)=a$ and $a\in \T_2$ that depends continuously on $a$.
\end{proposition}
\proof
Recall that the universal covering maps $\phi_a:\D\rightarrow L_a$
are obtained from each other by composing with a rotation on
$\D$. For the rest of the proof, denote by $\phi_a$ the universal
covering map such that in the coordinates on the flow box $\U$,
the derivative of $\phi_a$ at 0 is a positive real number. 
We can replace $\T_1$ with some compact set in the
support of $\nu$ in order to assume that
$\T_1$ is compact and contained in the support of $\nu$. 
We will use the notation in the previous proposition with
$K_0$ larger than $\T_1$. By Lusin's theorem, we can replace 
$\T_1$ by a suitable compact set in order to assume that $\xi_n$ and
$\vartheta$ are continuous on $\T_1$.
So, the sequence $\xi_n$ is increasing and converge uniformly to $\vartheta^{-1/2}$ on $\T_1$.

\medskip
\noindent
{\bf Claim.} Let $0<r<1$ and $\delta>0$ be two constants. 
Then, there is a compact set $\T_r\subset \T_1$ with
$\nu(\T_r)>\nu(\T_1)-\delta$ and an integer $N$ such that 
$\|D\phi_a\|\leq N$ on $r\D$ and $\phi_a(r\D)\subset K_N$
for all $a\in \T_r$. 

\medskip

We first explain how to deduce the proposition from the claim. Using this
property for $r_n=1-1/n$, $\delta_n=2^{-n}\epsilon$ with $n\geq 2$, define 
$\T_2:=\cap \T_{r_n}$. It is clear that
$\nu(\T_2)>\nu(\T_1)-\epsilon$ and the family $\{\phi_a\}_{a\in\T_2}$
is locally bounded on $\D$. If $a_n\rightarrow a$ in $\T_2$, since $\vartheta$ is
continuous, any limit value $\phi$ of $\phi_{a_n}$ satisfies
$\|D\phi(0)\|=\vartheta(a)^{-1/2}$. Therefore, $\phi$ is a universal covering
map of $L_a$. We deduce that $\phi$ is equal to $\phi_a$ because
the derivatives of 
$\phi_{a_n}$ and $\phi_a$ are
real positive. Hence, the family $\{\phi_a\}_{a\in\T_2}$ is continuous.

It remains to prove the above claim. 
Let $\Ec_n$ denote the family of $\tau\in \Fc_n$ such that
$a:=\tau(0)$ is in $\T_1$ and $\|D\tau(0)\|=\xi_n(a)$. This family is
not empty since $\Fc_n$ is compact. Let $\Ec_n^+$ be the family of
$\tau\in\Ec_n$ as above such that in the coordinates on the flow box $\U$,
the derivative of $\tau$ at 0 is a positive real number. We can obtain
such a map by composing a map in $\Ec_n$ with an appropriate rotation
on $\D$. The continuity of $\xi_n$ implies that $\Ec_n$ and $\Ec_n^+$ are compact.

The map which associates to $\tau\in\Ec_n^+$ its value at 0 is
continuous. 
We recall that if $f:X_1\rightarrow X_2$ is a continuous surjective map between
two compact metric spaces, then $f$ admits an  inverse measurable
selection, i.e. there is $g:X_2\rightarrow X_1$ measurable such that
$f\circ g$ is identity, \cite[p.82]{Christensen}. So, the map
$\tau\mapsto \tau(0)$ on $\Ec_n^+$ 
admits a measurable inverse map.
More precisely, there is a measurable family $\{\tau_{n,a}\}_{a\in\T_1}\subset
\Ec_n^+$ such that $\tau_{n,a}(0)=a$.
Therefore, the measure $\nu$ on $\T_1$ induces a measure on
$\Ec_n^+$. We can extract from $\{\tau_{n,a}\}_{a\in\T_1}$ a compact
subset of measure almost equal to $\nu(\T_1)$. Hence, there is a compact
set $\T_1'\subset \T_1$ such that $\nu(\T_1')>\nu(\T_1)-\delta$ and
$\{\tau_{n,a}\}_{a\in\T_1'}$ is compact for every $n$. In other
words, the family  $\{\tau_{n,a}\}_{a\in\T_1'}$ depends continuously
on $a\in \T_1'$.

For each $a$ fixed in $\T_1'$, the extremal property of Poincar{\'e}'s
metric implies that $\tau_{n,a}\rightarrow \phi_a$ locally
uniformly when $n\rightarrow\infty$. Define for all positive integer $N$, $\T_{1,N}$ the set of
$a\in\T_1'$ such that $\|D\tau_{n,a}\|\leq N$ on $r\D$ and
$\tau_{n,a}(r\D)\subset K_N$ for all
$n$. This is an increasing sequence of compact sets which converges
to $\T_1'$. So, for $N$ large enough,
$\nu(\T_{1,N})>\nu(\T_1)-\delta$. We can choose a compact subset
$\T_r\subset \T_{1,N}$ such that $\nu(\T_r)>\nu(\T_1)-\delta$.
Clearly, $\T_r$ satisfies the claim.
\endproof

Let $\phi_a:\D\rightarrow L_a$ be a covering map of $L_a$ with $\phi_a(0)=a$. Denote
$L_{a,R}:=\phi_a(\D_R)$, where $\D_R\subset \D$ is the disc of
center 0 and of radius $R$. Here, the radius is with respect to the
Poincar{\'e} metric on $\D$. Since $\phi_a$ is unique up to a
rotation on $\D$, $L_{a,R}$ is independent
of the choice of $\phi_a$. We will need the following result.

\begin{corollary} \label{cor_lusin}
Let $R>0$ be a positive constant.
Then, under the hypothesis of Proposition
  \ref{prop_lusine}, there is  a countable family of compact sets $\S_n\subset \T_1$, $n\geq1$, with
  $\nu(\cup_n \S_n)=\nu(\T_1)$ such that $L_{a,R}\cap \S_n=\{a\}$ for every $a\in \S_n$. Moreover, there are universal covering maps $\phi_a:\D\rightarrow L_a$ with $\phi_a(0)=a$ 
which depend continuously on $a\in\S_n$.
\end{corollary}
\proof
We first show that there is a compact set $\S\subset \T_2$ with $\nu(\S)>0$ such that
$L_{a,R}\cap \S=\{a\}$ for every $a\in \S$. By Proposition \ref{prop_lusine}, the last assertion in the corollary holds for $\S$. 
Consider $\T_2'$ the support of the restriction of $\nu$ to $\T_2$ and an open neighbourhood $\T'$
of $\T_2$ which is relatively compact in $\T$. By Proposition
\ref{prop_lusine}, the number of points in $\Sigma_a:=L_{a,R}\cap\T'$ is bounded independently on $a\in\T_2'$ because the minimal number plaques covering $\overline L_{a,R}$ is bounded.  
Fix an $a_0\in \T_2'$ such that 
$\#\Sigma_{a_0}$ is maximal. Also by Proposition \ref{prop_lusine}, 
$\#\Sigma_a$ is lower semi-continuous on
$a\in \T_2'$. The maximality of $\#\Sigma_{a_0}$ implies that if $V$
is a neighbourhood of $a$, small enough, $\#\Sigma_a=\#\Sigma_{a_0}$
for $a\in \T_2'\cap \overline V$. It follows that $\Sigma_a$ depends
continuously on $a\in \T_2'\cap \overline V$. We then deduce that if $V$ is small enough,
$\Sigma_a\cap \overline V=\{a\}$ for  $a\in \T_2'\cap \overline V$. It is enough
to take $\S:=\T_2'\cap\overline V$; we have $\nu(\S)>0$ by definition of $\T_2'$. 

Consider now the family $\Gc$ of all countable unions $G$ of such compact sets $\S$. 
Let $\lambda$ denote the supremum of $\nu(G)$ for $G\in\Gc$. So, there is a sequence $G_n$ in $\Gc$ 
such that $\nu(G_n)\to \lambda$.  The union $G_\infty:=\cup_n G_n$ is also an element of $\Gc$. So, we have $\nu(G_\infty)=\lambda$. Now, it is enough to check that $\lambda=\nu(\T_1)$. If not, we have $\nu(\T_1\setminus G_\infty)>0$. Hence, we can apply the above construction of $\S$ in $\T_1\setminus G_\infty$ instead of $\T_1$. We necessarily have  $\nu(G_\infty\cup\S)>\lambda$ which is a contradiction. So, we can choose compact sets $\S_n$ satisfying the corollary with $\cup \S_n=G_\infty$.
\endproof

\section{Laminations with singularities} \label{section_singular}

We call {\it Riemann surface lamination with singularities} the data
$(X,\Lc,E)$ where $X$ is a locally compact space, $E$ a closed
subset of $X$ and $(X\setminus E,\Lc)$ is a Riemann surface
lamination. The set $E$ is {\it the singularity set} of the lamination.
In order to simplify the presentation, we will mostly consider the
case where $X$ is a closed 
subset of a
complex manifold $M$ of dimension $k\geq 1$ and $E$ is a locally finite subset
of $X$. 
We assume that $M$ is endowed with a Hermitian metric $\omega$. 
We also assume that the complex structures on
the leaves of the foliation coincide with the ones induced by $M$,
that is, the leaves of $(X\setminus E,\Lc)$ are Riemann surfaces holomorphically immersed in
$M$. 
The main example we have in mind is a foliation by Riemann surfaces in
the projective space $\P^k$ described in the introduction.

\begin{proposition} \label{prop_current_extension}
Let $(X,\Lc,E)$ be a lamination with isolated singularities
in a complex manifold $M$ as above. Let $T$ be a positive harmonic
current of $X\setminus E$. Then the linear form
$\alpha\mapsto \langle T,\alpha_{|X}\rangle$ for  
$\alpha\in \Dc^{1,1}(M\setminus E)$
defines a positive $\ddbar$-closed current on $M\setminus
E$. Moreover, it has locally finite mass on $M$ and when $X$ is compact, the extension of $T$ by zero, always denoted by $T$, is a
positive $\ddbar$-closed current of bidimension $(1,1)$ on $M$.
\end{proposition}
\proof
Observe that $\alpha_{|X}$ is smooth with compact support in
$X\setminus E$ which is positive (resp. $\ddbar$-exact) if $\alpha$ is
positive (resp. $\ddbar$-exact). Therefore, using a partition of unity and the local description of $T$, we see that 
$T$ defines a positive $\ddbar$-closed current on $X\setminus E$. 
Since $E$ is finite, $T$ has locally finite mass on $M$, see \cite{FSW}. If moreover $X$ is compact, the
extension of $T$ by zero is a positive $\ddbar$-closed current on $M$, see Lemma \ref{lemma_ds}. 
\endproof

So, if $T$ is a positive harmonic current on $X\setminus E$, its mass with
respect to the Hermitian metric on $M$ is locally finite. 
We call {\it Poincar{\'e}'s mass} of $T$ the mass of $T$ with respect to
Poincar{\'e}'s metric $\omega_P$ on $X\setminus E$, i.e. the mass of the
positive measure $m_P:=T\wedge\omega_P$.
A priori, Poincar{\'e}'s mass may be
infinite near the singular points. The following proposition
gives us a criterion for the finiteness of this mass. It can be applied
to generic foliations in $\P^k$.

We say that a vector field $F$ on $\C^k$ is {\it generic linear} if
it can be written as 
$$F(z)=\sum_{j=1}^k \lambda_j z_j {\partial\over \partial z_j}$$
where $\lambda_j$ are non-zero complex numbers.
The integral curves of $F$ define a Riemann surface foliation on
$\C^k$. 
The condition $\lambda_j\not=0$ implies that the foliation   
has an isolated singularity at 0. Consider a lamination $X$ with isolated
singularities $E$ in a manifold $M$ as above. We say that a
singular point $a$ of $X$ is {\it linearizable} if 
there is a local holomorphic coordinates system of $M$ near $a$ on which the
leaves of $X$ are integral curves of a generic linear vector field.

\begin{proposition} \label{prop_poincare_mass}
Let $(X,\Lc,E)$ be a compact lamination with isolated singularities
in a complex manifold $M$. If $a$ is a linearizable singularity of $X$, 
then any positive harmonic current on $X$ has locally finite
Poincar{\'e}'s mass near $a$.
\end{proposition}

Fix a positive harmonic current $T$ on $X$. So, by  Proposition \ref{prop_current_extension}, we can identify $T$  with a positive $\ddbar$-closed current on $M$. For the rest of the proof, 
we don't need  the compactness of $M$.
Since the Poincar{\'e} metric increases when we replace $M$ with an open
subset, it is enough to consider the case where $M$ is the
polydisc $(2\D)^k$ in $\C^k$ and $X$ is the lamination associated with the
vector field 
$$F(z)=\sum_{j=1}^k \lambda_jz_j{\partial\over \partial z_j}$$
where $\lambda_j=s_j+it_j$ and $s_j,t_j\in\R$. 
We need the following lemma.

\begin{lemma} \label{lemma_poincare}
For every point $a\in \D^k\setminus \{0\}$, 
there is a holomorphic map $\tau:\D\rightarrow L_a\cap (2\D)^k$ such that
$\tau(0)=a$ and $\|D\tau(0)\|\geq c\|a\||\log\|a\||$ where $c>0$ is a
constant independent of $a$.
\end{lemma}
\proof
We only have to consider $a$ very close to 0. 
Let $\psi_a:\C\rightarrow\C^k\setminus\{0\}$ be the holomorphic map defined by
$$\psi_a(\xi):=\Big(a_1e^{\lambda_1\xi},\ldots,a_ke^{\lambda_k\xi}\Big)\quad \mbox{for}\quad \xi\in\C$$
where $a_j$ are the coordinates of $a$. We have $\psi_a(0)=a$ and $\psi_a(\C)$ is an integral curve of $F$. 
Write $\xi=u+iv$ with $u,v\in\R$. The domain $\psi_a^{-1}(\D^k)$ in $\C$ is defined by the inequalities
$$s_ju-t_jv\leq -\log|a_j| \quad \mbox{for}\quad j=1,\ldots,k.$$
So, $\psi_a^{-1}(\D^k)$ is a convex polygon (not necessarily
bounded)
which contains $0$ since $\psi_a(0)=a\in \D^k$. Observe that the 
distance between 0 and the line $s_ju-t_jv=-\log|a_j|$ is proportional to $-\log|a_j|$. Therefore, 
$\psi_a^{-1}(\D^k)$ contains a disc of center 0 and of radius 
$$c'\min\{-\log|a_1|,\ldots, -\log|a_k|\}\geq -c'\log\|a\|$$
for some constant $c'>0$ independent of $a$. 

Define the map $\tau:\D\rightarrow \C^l$ by
$\tau(\xi):=\psi_a(-c'\log\|a\|\xi)$. It is clear that $\tau(0)=a$ 
and $\tau(\D)\subset L_a\cap\D^k$. 
We also have
$$\|D\tau(0)\|=-c'\log\|a\| \|D\psi_a(0)\|\geq -c'\log\|a\|\min_j |\lambda_j|\|a\|.$$
The lemma follows.
\endproof

\noindent
{\bf Proof of Proposition \ref{prop_poincare_mass}.}
We use in $\C^k$ the standard K{\"a}hler metric $\beta:=i\ddbar\|z\|^2$. 
Recall that $\omega_P=4\vartheta\beta$. Lemma \ref{lemma_poincare} implies
that $\vartheta(a)\lesssim \|a\|^{-2}|\log\|a\||^{-2}$.
Let $T$ be a positive $\ddbar$-closed
current on $(2\D)^2$. We only have to show that the integral on the
the ball $B_{1/2}$, with respect to the measure $T\wedge\beta$, of
the radial function 
$\widetilde\vartheta(r):=r^{-2}|\log r|^{-2}$ is finite. 
Let $m(r)$ denote the mass of $T\wedge \beta$ on $B_r$. 
By Lemma \ref{lemma_skoda}, we have $m(r)\lesssim r^2$. 
Using an integration by parts, the considered integral is equal, up to finite constants, to
$$-\int_0^1 m(r)\widetilde\vartheta'(r)dr$$
It is clear that the last integral is finite. The proposition follows. 
\hfill $\square$

\bigskip

We have few remarks. It is shown in Candel-G{\'o}mez-Mont\cite{CandelGomezMont}, see also \cite{FornaessSibony2} that when $a$ is a hyperbolic singularity, $\vartheta(z)\to \infty$ when $z\to a$.  
We only consider in
this paper Poincar{\'e}'s metric on the regular part of the
lamination. It is quite often that some leaves of the lamination can
be compactified near singular points by adding these points
and sometimes it is natural to
consider the Poincar{\'e} metric of the extended leaves. Since
Poincar{\'e}'s metric decreases when we extend these leaves, several
results we obtain also apply to extended leaves as well. 

If $\T$ is a transversal in $X$ then all positive harmonic currents on $(X,\Lc,E)$ have finite mass near $\T$ with respect to the Poincar{\'e} metric on the leaves of $X\setminus (E\cup\T)$. This may give us a technical tool in order to study parabolic leaves. Another situation where we have currents with finite Poincar{\'e} mass is the following. Let $\pi:M'\rightarrow M$ be a proper finite holomorphic map and $(X',\Lc',E')$ 
be a compact Riemann surface lamination on $M'$ with isolated singularities which is the pull-back of a lamination $(X,\Lc,E)$ with linearizable singularities as above. If $T'$ is a positive harmonic current on $M'$, then its Poincar{\'e} mass is bounded by the 
Poincar{\'e} mass of the positive harmonic current $\pi_*(T')$ of $X$ since $\pi$ contracts the Poincar{\'e} metric.

We have the following properties of positive harmonic currents, see also \cite[Th. 3.14]{FornaessSibony1} for the last assertion.
 
\begin{proposition} \label{prop_choquet_c}
Let $(X,\Lc,E)$ be a compact Riemann surface 
lamination with isolated singularities in a Hermitian complex manifold $(M,\omega)$. 
Let $\Gc$ be the family of positive harmonic currents of mass $1$ on $X$.
Then, $\Gc$ is a non-empty compact simplex and for any $T\in\Gc$ there 
is a unique probability measure $\nu$ on the set of extremal elements in $\Gc$ such that 
$T=\int T' d\nu(T')$. Moreover, two different extremal elements in $\Gc$ are mutually singular.
\end{proposition}
\proof
Recall that {\it the mass} of $T$ is the mass of the measure $T\wedge\omega$.
The fact that $\Gc$ is compact convex is clear. The existence of positive harmonic currents was obtained in \cite{BerndtssonSibony}. 
By Choquet's representation theorem \cite{Choquet}, we can decompose $T$ into extremal elements as in the proposition. We show that the decomposition is unique. According to Choquet-Meyer's theorem \cite[p.163]{Choquet}, it is enough to check that the cone generated by $\Gc$ is a lattice. More precisely, there is a natural order in this cone: $T_1\prec T_2$ if $T_2-T_1$ is in the cone. We have to show that given two elements $T_1,T_2$ in the cone there is a minimal element $\max\{T_1, T_2\}$ larger than $T_1,T_2$ and a maximal element $\min\{T_1,T_2\}$ smaller than $T_1,T_2$ with respect to the above order.

Define $T:=T_1+T_2$. Using the description of currents in flow boxes, we see that   $T_i=\theta_iT$ for some functions $\theta_i$ in $L^1(T\wedge\omega)$, $0\leq \theta_i\leq 1$. 
Define also $\max\{T_1,T_2\}:=\max\{\theta_1,\theta_2\}T$ and $\min\{T_1, T_2\}:=\min\{\theta_1,\theta_2\}T$.
It is enough to show that these currents belong to the cone generated by $\Gc$. 
In a flow box $\U\simeq \B\times\T$ as above, if $\mu$ is the transversal measure associated to $T$, then $T_i$ is given by harmonic functions $h_{i,a}$ on the plaques $\B\times\{a\}$ for $\mu$-almost every $a$. The current $\min\{T_1, T_2\}$ is associated with $\min\{h_{1,a},h_{2,a}\}$ which is positive superharmonic. Therefore, $\min\{T_1, T_2\}$ is a positive current on $M$ and $\ddc \min\{T_1, T_2\}$ is a negative measure on $M\setminus E$. By Lemma \ref{lemma_ds}, $\min\{T_1,T_2\}$ is harmonic and is an element of $\Gc$.  
It follows that $\max\{T_1,T_2\}=T-\min\{T_1,T_2\}$ is also an element of $\Gc$. This completes the proof of the first assertion in the proposition.

Let $T,T'$ be two different extremal elements in $\Gc$. We show that they are mutually singular.
Using the local description of currents, we can write $T'=\theta T+T''$ where $\theta$ is a positive function in $L^1(T\wedge\omega)$ and $T\wedge\omega$, $T''\wedge\omega$ are mutually singular.  Using also the local description of currents, we see that $\theta T$ and $T''$ are necessarily harmonic. 
There is a union $\Sigma$ of leaves such that $T''$ has no mass outside $\Sigma$ and $T$ has no mass on $\Sigma$. If $T''$ is non-zero, since $T'$ is extremal, it has no mass outside $\Sigma$ and then $T,T'$ are mutually singular. Assume that $T''=0$.

We can find a number $c>0$  such that $\{\theta\geq c\}$ and $\{\theta\leq c\}$ have positive measure with respect to $T\wedge\omega$. Define $T^+_c:=\max\{\theta-c,0\}T$, $T^-_c:=\max\{c-\theta,0\}T$.
Since $T^+_c=\max\{T',cT\}-cT$ and $T^-_c=\max\{cT,T'\}-T'$, these currents are harmonic.
So, we can choose a set $\Sigma'$ which is a union of leaves such that  $T_c^+$ has no mass outside $\Sigma'$ and $\theta> c$  on $\Sigma'$. The choice of $c$ implies that
 $T$  has positive mass outside $\Sigma'$. 
 Since $T$ is extremal, we deduce that $T$ has no mass on $\Sigma'$.
It follows that $T_c^+=0$ and then $\theta\leq c$ almost everywhere. Using $T_c^-$, we prove in the same way that $\theta\geq c$ almost everywhere. Finally, we have $T'=cT$ and since $T,T'$ have the same mass we get $T=T'$. This is a contradiction.
\endproof

Recall that a leaf  $L$  in  $(X,\Lc,E)$ is  {\it wandering} if it is not closed in $X\setminus E$ and if 
there is a point $p\in L$ and a  flow box  $\U$  containing $p$ such
that $L\cap \U$  
is  just one plaque. Note that if $L$ is wandering, the above property is true for every $p\in L$.
The set of closed leaves and the set of wandering leaves are measurable.
We  get  as for real smooth foliations \cite[p.113]{CandelConlon2}  the  following result.

\begin{theorem}
Let $M$ be a complex manifold and  $\omega$ a Hermitian form on $M$.
Let $(X,\Lc,E)$  be   a compact lamination with isolated singularities  in $M$ 
 and $T$ a positive harmonic current on $X$. Then the  set of wandering  leaves of $X$ has zero
measure  with respect  to the measure $T\wedge\omega$.    
\end{theorem}
\proof
Assume that the set $Y$ of wandering leaves has positive $T\wedge\omega$ measure. Choose a flow box 
$\U\simeq \B\times \T$ such that $Y\cap U$ has positive measure. 
For each $m\geq 1$, cover $\T$ by a finite family of open sets $\S_{m,n}$ of diameter $\leq 1/m$. 
Denote by $Y_{m,n}$ the set of leaves which intersect $\S_{m,n}$ at only one point. So, the union of $Y_{m,n}$ has positive measure since it contains $Y\cap\U$. We can choose  
$m,n$ such that $Y_{m,n}$ has positive measure.

By Proposition \ref{prop_current_local_c}, the restriction of $T$ to $Y_{m,n}$ is a harmonic current. So, we can restrict $T$ to $Y_{m,n}$ in order to 
assume that $T$ has no mass outside $Y_{m,n}$. We deduce from the definition of $Y_{m,n}$ that $T$ can be decomposed into extremal harmonic currents supported on one single leaf of $Y_{m,n}$. 
We can assume that $T$ is extremal. So, there is 
a leaf $L$ in $Y_{m,n}$ and a positive harmonic function $h$ such that $T=h[L]$.  
We show that $L$ is closed in $X\setminus E$ and this will give us a contradiction. For this purpose, one only has to prove that $L$ has finite area. 

Observe that $h$ is integrable on $L$ with respect to the metric $\omega$ since it defines a positive 
current on $M$. 
Therefore, it is enough to show that $h$ is constant. Assume that $h$ is not constant. Let $t<s$ be two different values of $h$.  Let $\chi:\R_+\rightarrow\R_+$ be a smooth concave increasing function such that $\chi(x)=x$ if $x\leq t$ and $\chi(x)=s$ if $x\geq s+1$. So, $\chi(h)\leq h$ and then $\chi(h)$ is integrable on $L$. Moreover, we have since $h$ is harmonic 
$$i\ddbar \chi(h)=\chi'(h)i\ddbar h + \chi''(h)i\partial h\wedge\dbar h=\chi''(h)i\partial h\wedge\dbar h\leq 0.$$
We deduce using flow boxes that $T':=\chi(h)[L]$ is a positive current with compact support in $M$ 
such that $i\ddbar T'\leq 0$ on $M\setminus E$. By Lemma \ref{lemma_ds}, 
$i\ddbar T'=0$ and then $\chi(h)$ is harmonic. Finally, since $\chi(h)$ has a maximum on $L$, it is constant. This gives a contradiction.  
\endproof

\section{Heat  equation  on a real lamination} \label{section_heat_real}

Consider a real lamination $(X,\Lc)$ of dimension $n$ as in Section
\ref{section_current}. We assume that $X$ is compact and endowed with
a continuously smooth Riemannian
metric $g$ on the leaves. By continuous smoothness, we mean that in flow boxes the
coefficients of $g$ and their derivatives of any order depend
continuously on the plaques. 
So, we  can consider
the corresponding Laplacian $\Delta$ and the gradient $\nabla$ on
leaves. Several results below for $\Delta$ can be deduced from
Candel-Conlon \cite{CandelConlon2} and Garnett \cite{Garnett}  but
results on $\widetilde\Delta$ are new. 
Observe that in our context $\Delta$ is not symmetric in $L^2(m)$. 

A measure $m$ on $X$ is called {\it $g$-harmonic} (or simply {\it
  harmonic} 
when there is no confusion) 
if
$$ \int  \Delta u dm=0 \quad \mbox{for}\quad u\in\Dc^0(X).$$
We will consider the operator $\Delta$ in $L^2(m)$ and the
above identity holds for $u$ in the domain 
$\Dom(\Delta)$ of $\Delta$ that we will define later. We have the following elementary result.

\begin{proposition} \label{prop_choquet}
Let $\Omega$ be the volume form associated with $g$. 
Then there is a one to one correspondence between harmonic (resp. positive
harmonic) measures $m$ and harmonic (resp. positive harmonic) 
current $T$ such that $m=T\wedge\Omega$. 
\end{proposition}
\proof
It is clear that $T\wedge\Omega$ is a harmonic (positive) measure if $T$ is a harmonic (positive) 0-current. 
Consider now a decomposition of $m$ in a flow box as in Proposition
\ref{prop_struc_current}
$$m=\int m_a d\mu(a)$$
with $m_a$ a measure on $\B\times\{a\}$. Since $m$ is harmonic, we
deduce that $m_a$ is harmonic in the flow box for $\mu$-almost every
$a$. So, by Weyl's lemma, there is a harmonic function $h_a$
such that $m_a=h_a\Omega$. We have $m=T\wedge\Omega$ where $T$ is
locally given by
$$T=\int h_a [\B\times\{a\}] d\mu(a),$$
where $[\B\times\{a\}]$ is the current of integration on the plaque
$\B\times\{a\}$. 
When $m$ is positive, it is easy to see that $h_a$ is positive for $\mu$-almost every $a$. The result follows.
\endproof

Consider a harmonic probability measure $m$ on $X$. Write
$m=T\wedge\Omega$ as in 
Proposition \ref{prop_choquet}. In a flow box
$\U\simeq \B\times\T$, by Proposition \ref{prop_struc_current}, the current $T$ can be
written as 
$$T=\int h_a [\B\times\{a\}]d\mu(a),$$
where $h_a$ is a  positive harmonic function on $\B$  and $\mu$ is
a  positive  measure  on the transversal $\T$. We can restrict $\mu$
in order to assume that $h_a\not=0$ for $\mu$-almost every $a$. Under
this condition, Proposition \ref{prop_struc_current_unique} says that
up to a multiplicative function, $\mu$ and $h_a$ are
uniquely determined by $T$.

In what follows, the differential operators $\nabla$, $\Delta$ and $\widetilde\Delta$ are considered in $L^2(m)$. We introduce the  Hilbert  space $H^1(m)$ as
the completion of $\Dc^0(X)$ with  respect to  the  norm
$$ \|u\|_{H^1}^2:=\int  |u|^2 dm+\int | \nabla u|^2 dm.$$
Recall that the gradient $\nabla$ is defined by 
$$\langle \nabla u,\xi\rangle_g=du(\xi)$$
for all tangent vector $\xi$ along a leaf and for $u\in\Dc^0(X)$. 
We consider $\nabla$ as a operator in $L^2(m)$ and $H^1(m)$ its domain.

Define in a flow box $\U\simeq \B\times \T$ as above the Laplace type operator
$$-\widetilde \Delta u= -\Delta u-\langle h_a^{-1}\nabla h_a,\nabla u\rangle_g=-\Delta u-F u,$$ 
where $F$ is a vector field.
The uniqueness of $h_a$ and $\mu$ implies that $F$ does
not depend on the choice of the flow box. Therefore, $Fu$ and
$\widetilde\Delta u$ are defined globally $m$-almost everywhere when $u\in\Dc^0(X)$.

We recall some classical results of functional analysis that we will
use later. The reader will find the details in Brezis \cite{Brezis}. A
linear operator $A$ on a Hilbert space $L$ is called {\it monotone} if
$\langle Au,u\rangle\geq 0$ for all $u$ in the domain $\Dom(A)$ of
$A$. Such an operator is {\it maximal monotone} if moreover for any
$f\in L$ there is a $u\in\Dom(A)$ such that $u+Au=f$. In this case,
the domain of $A$ is always dense in $L$ and the graph of $A$ is closed. 

A family $S(t):L\rightarrow L$, $t\in\R_+$, is {\it a semi-group of
  contractions} if $S(t+t')=S(t)\circ S(t')$ and if $\|S(t)\|\leq 1$
for all $t,t'\geq 0$. We will apply the following theorem 
to our Laplacian operators and for $L:=L^2(m)$. It says that any maximal 
monotone operator is the infinitesimal generator of a semi-group of contractions.

\begin{theorem}[Hille-Yosida] \label{th_hille_yosida}
Let $A$ be a maximal monotone operator on a Hilbert space $L$. Then
there is a semi-group of contractions $S(t):L\rightarrow L$,
$t\in\R_+$, such that for $u_0\in\Dom(A)$, $u(t,\cdot):=S(t)u_0$ is
the unique $\Cc^1$ function from $\R_+$ to $L$ with values in
$\Dom(A)$ which satisfies
$${\partial u(t,\cdot)\over \partial t}+Au(t,\cdot)=0\quad
\mbox{and}\quad u(0,\cdot)=u_0.$$ 
When $A$ is self-adjoint and $u_0\in L$, the function $u(t,\cdot)$ is
still continuous on $\R_+$ and is $\Cc^1$ on $\R_+^*$ with values in
$\Dom(A)$ and we have the estimate
$$\Big\|{\partial u\over\partial t}\Big\|\leq {1\over t}\|u_0\|\quad \mbox{for } t>0.$$ 
\end{theorem}

In order to check that our operators are maximal monotone, we will
apply the following result to $H:=H^1(m)$.

\begin{theorem}[Lax-Milgram] \label{th_lax_milgram}
Let $e(u,v)$ be a continuous bilinear form on a Hilbert space
$H$. Assume that $e(u,u)\geq \|u\|^2_H$ for $u\in H$. Then for every $f$
in the dual $H^*$ of $H$ there is a unique $u\in H$ such that 
$e(u,v)=\langle f,v\rangle$ for $v\in H$. 
\end{theorem}

Define for $u,v\in\Dc^0(X)$
$$q(u,v):=- \int (\Delta u)v dm, \qquad
e(u,v) := q(u,v)  +\int uvdm$$
and
$$\widetilde q(u,v):=-\int (\widetilde{\Delta}u)v dm=q(u,v)-\int
(Fu)vdm, \qquad
\widetilde e(u,v):=\widetilde q(u,v)+\int uvdm.$$
Note that these identities still hold for $v\in L^2(m)$ and $u$ in the
domain of $\Delta$ and of $\widetilde\Delta$ that we will
define later.

\begin{lemma} \label{lemma_e_tilde}
We have 
for $u,v\in\Dc^0(X)$
$$\widetilde q(u,v)=\int \langle\nabla u,\nabla v\rangle_g dm  
\quad \mbox{and}\quad \int (\widetilde\Delta u)v dm = \int
u(\widetilde\Delta v) dm.$$
In particular, $\widetilde q(u,v)$ and $\widetilde e(u,v)$ are symmetric in $u,v$ and 
$$\int \widetilde\Delta u dm=\int\Delta u dm=\int Fu dm=0 \quad \mbox{for}\quad u\in\Dc^0(X).$$
\end{lemma}
\proof
Using a partition of unity, we can assume that $u$ and $v$ have
compact support in a flow box as above. It is then enough to consider
the case where $T$ is supported by a plaque $\B\times\{a\}$ and given
by a harmonic function $h_a$. The first identity in the lemma can be
deduced from classical identities in Riemannian geometry, see e.g. \cite{Chavel}. Indeed, we have
$$\langle \nabla u,\nabla v\rangle_g+v\Delta u=\div (v\nabla u)$$
and the first identity is equivalent to
$$\int_\B\div(v\nabla u)h_a\Omega= -\int_\B (Fu)v h_a\Omega$$
which is easily obtained by integration by parts.

It follows from the first identity in the lemma that $\widetilde q$
and $\widetilde e$ are
symmetric. The second assertion in the lemma is a consequence.
We then deduce the last identities in the lemma 
using that $\widetilde\Delta 1=0$ and that $m$ is harmonic. Note that the lemma still holds for
$u,v$ in the domain of $\Delta$ and $\widetilde\Delta$ that we will
define later.
\endproof

\begin{lemma} \label{lemma_e_e_tilde}
The bilinear forms $q$, $\widetilde q$, $e$ and $\widetilde e$ extend
  continuously to $H^1(m)\times H^1(m)$. Moreover, we have
$q(u,u)=\widetilde q(u,u)$ and $e(u,u)=\widetilde e (u,u)$ for $u\in H^1(m)$. 
\end{lemma}
\proof
The first identity in Lemma \ref{lemma_e_tilde} implies that
$\widetilde q$ and $\widetilde e$
extend continuously to $H^1(m)$ and the identity is still valid for
the extension of $\widetilde q$. 
In order to prove the same property
for $q$ and $e$, it is enough to show that $q-\widetilde q$ is bounded on
$H^1(m)\times H^1(m)$. 

We use the description of $\widetilde\Delta$ in a flow box $\U$ as
above. By Harnack's inequality, $h_a^{-1}\nabla h_a$ 
are locally bounded uniformly on $a$. Since $X$ is compact, we deduce that
the  vector field $F$ has  bounded coefficients.
Therefore, by Cauchy-Schwarz's inequality 
$$|q(u,v)-\widetilde q(u,v)|=\Big|\int (Fu) v dm\Big| \leq  C \Big(\int|\nabla u|^2 dm\Big)^{1/2}
 \Big(\int v^2 dm\Big)^{1/2}.$$ 
So, $q-\widetilde q$ is bounded and hence $q,e$ extend continuously on
$H^1(m)\times H^1(m)$. 

We prove now the last identity in the lemma. We can assume that $u$ is
in $\Dc^0(X)$ since this space is dense in $H^1(m)$. We have
$$\Delta(u^2)=\div(\nabla u^2)=\div(2u\nabla u)= 2u\Delta u + 2|\nabla
u|^2.$$
Since $m$ is harmonic, the integral of $\Delta(u^2)$ vanishes. It
follows that
$$-\int u\Delta u dm =\int|\nabla u|^2dm.$$
We deduce from the definition of $q,e$ and the first identity in Lemma
\ref{lemma_e_tilde} that $q(u,u)=\widetilde q(u,u)$ and
$e(u,u)=\widetilde e(u,u)$. 
This completes the proof.
\endproof

Define the domain $\Dom(\Delta)$ of $\Delta$
(resp. $\Dom(\widetilde\Delta)$ of $\widetilde\Delta$) as the space of $u\in
H^1(m)$ such that $q(u,\cdot)$ (resp. $\widetilde q(u,\cdot)$) extends to a linear
continuous form on $L^2(m)$. 
Since $\Delta-\widetilde\Delta$ is given
by a vector field with bounded coefficients, we have
$\Dom(\Delta)=\Dom(\widetilde\Delta)$. In a flow box, we can show using Federer's version of Lusin's theorem \cite[Th. 2.3.5]{Federer} that a function $u$ in $L^2(m)$ belongs to $H^1(m)$ if the gradient $\nabla u$, defined as a vector field with distribution coefficients on generic plaques, is in $L^2(m)$. An analogous property holds for $\Delta$ and $\widetilde\Delta$. 
In fact, $\Dom(\Delta)$ is the completion of $\Dc^0(X)$ for the norm 
$\sqrt{\|u\|^2_{L^2(m)}+\|\Delta u\|_{L^2(m)}^2}$. 
It is clear that if $u\in\Dom(\Delta)$ then $\Delta u$ in the sense of distributions with respect to $\Dc^0(X)$ as test functions, is in $L^2$. This allows us to extend Lemma \ref{lemma_e_tilde} to $u,v$ in $\Dom(\Delta)$. 
The converse is true but we don't use it.
We have the following proposition.

\begin{proposition} \label{prop_delta_max}
Let $(X,\Lc)$ be a compact real lamination of dimension $n$ endowed
with a continuously smooth Riemannian metric $g$ on the leaves. Let $m$ be
a harmonic probability measure on $X$.
Then the associated operators $-\Delta$ and
$-\widetilde\Delta$ are maximal monotone on $L^2(m)$. In particular,
they are infinitesimal generators of semi-groups of contractions on $L^2(m)$ and their graphs are closed.
\end{proposition}
\proof
The last assertion is the consequence of the first one, Theorem
\ref{th_hille_yosida} and the properties of maximal monotone operators. 
So, we only have to prove the first assertion.
We deduce from Lemmas \ref{lemma_e_tilde} and \ref{lemma_e_e_tilde}
that for $u\in \Dc^0(X)$
$$\langle - \Delta u ,u \rangle  =\langle - \widetilde \Delta u ,u \rangle  
=\int \langle \nabla u,\nabla u\rangle_g  dm=\|\nabla u\|_{L^2}^2\geq 0.$$
By continuity, we can extend the inequalities to $u$ in $\Dom(-\Delta)=\Dom(-\widetilde\Delta)$.
So, $-\Delta$ and $-\widetilde\Delta$ are monotone. We also obtain
for $u\in H^1(m)$ that
$$e(u,u)=\widetilde e(u,u)\geq \|u\|_{H^1}^2.$$
By Theorem \ref{th_lax_milgram}, for any $f\in L^2(m)$, there is $u\in
H^1(m)$ such that
$$e(u,v)=\langle f,v\rangle_{L^2(m)} \quad \mbox{for}\quad v\in
H^1(m).$$
So, $u$ is in $\Dom(\Delta)$ and the last equation is equivalent
to $u-\Delta u=f$. 
Hence, $-\Delta$ is maximal monotone. The case of $-\widetilde\Delta$ 
is treated in the same way. Note that since $-\widetilde\Delta$ is
symmetric and maximal monotone, it is self-adjoint but $-\Delta$ is not symmetric. 
\endproof

We have the following theorem.

\begin{theorem} \label{th_heat_real} 
Under the hypothesis of Proposition \ref{prop_delta_max}, let $S(t)$,
$t\in\R_+$, denote  the semi-group of contractions associated with the operator
$-\Delta$ or  $-\widetilde\Delta$ which is given by the Hille-Yosida
theorem. Then the  
measure $m$ is  $S(t)$-invariant and
$S(t)$ is a positive contraction in $L^p(m)$ for all $1\leq p\leq\infty$. 
\end{theorem}
\proof
We prove that $m$ is invariant, that is 
$$\langle m,S(t)u_0\rangle = \langle m,u_0\rangle \quad\mbox{for}\quad u_0\in \Dc^0(X).$$
We will see later that this identity holds also for $u_0\in L^1(m)$
because $S(t)$ is a contraction in $L^1(m)$ and $\Dc^0(X)$ is dense in $L^1(m)$. Define 
$u:=S(t)u_0$ and 
$$\eta(t):=\langle m,S(t)u_0\rangle=\langle m,u(t,\cdot)\rangle.$$
We deduce from Theorem \ref{th_hille_yosida} that $\eta$ is of class $\Cc^1$ on $\R_+$ and that
$$\eta'(t)=\langle m,S'(t)u_0\rangle =\langle m,Au(t,\cdot)\rangle$$
where $A$ is the operator $-\Delta$ or $-\widetilde\Delta$. By Lemma \ref{lemma_e_tilde},
the last integral vanishes. So, $\eta$ is constant and hence $m$ is invariant.

In order to prove the positivity of $S(t)$, it is enough to show the following {\it maximum principle}: if $u_0$ is a function in $\Dc^0(X)$ such that $u_0\leq K$ for some constant $K$, then $u(t,x)\leq K$. 
To show  the maximum  principle  we  use a trick due to  Stampacchia \cite{Brezis}.  Fix a smooth bounded function
 $G:\R\rightarrow \R_+$ with bounded first derivative such that  $G(t)=0$ for  $t\leq 0$
and   $G'(t)>0$ for  $t>0.$  Put
$$H(s):=\int_0^s G(t) dt. $$
Consider  the  non-negative  function $\xi:\R_+\rightarrow \R_+$  given by 
$$\xi(t):=\int H(u(t,\cdot)-K)dm.$$
By Theorem \ref{th_hille_yosida}, $\xi$ is of class $\Cc^1$. 
We  want to show that it is  identically zero. Define $v(t,x):=u(t,x)-K$. We have $Av(t,x)=Au(t,x)$. Using in particular that $G$ is bounded, 
we obtain
\begin{eqnarray*}
\xi'(t)&=&\int  G(u(t,\cdot)-K)\frac{\partial u(t,\cdot)}{\partial t} dm\\
&=&-\int  G(u(t,\cdot)-K) A u(t, \cdot) dm\\
&=&-\int  G(v(t,\cdot))Av(t,\cdot)dm.
\end{eqnarray*}
When $A=-\widetilde \Delta$, by Lemma \ref{lemma_e_tilde}, the last integral is equal to
$$-\int \langle \nabla  G(v),\nabla v\rangle_g dm =-\int G'(v)|\nabla v|^2dm\leq 0.$$
Thus, $ \xi'(t)\leq  0$. We deduce that $\xi=0$ and hence
$u(x,t)\leq  K$. 

When $A=-\Delta$, since $F=\widetilde\Delta-\Delta$ is a vector field
with bounded coefficients, the considered integral is equal to
$$-\int G'(v)|\nabla v|^2dm+\int G(v)Fvdm=-\int G'(v)|\nabla v|^2dm+\int FH(v)dm.$$
By Lemma \ref{lemma_e_tilde}, the last integral vanishes.
So, we also obtain that $\xi'(t)\leq 0$. This completes the proof of
the maximum principle which implies the positivity of $S(t)$. 

The positivity of $S(t)$ together with the invariance of $m$ imply that 
$$\|S(t)u_0\|_{L^1(m)}\leq \|u_0\|_{L^1(m)} \quad \mbox{for} \quad u_0\in\Dc^0(X).$$
It follows that $S(t)$ extends continuously to a positive contraction in $L^1(m)$ since $\Dc^0(X)$ is dense in $L^1(m)$. 
The uniqueness of the solution in  Theorem \ref{th_hille_yosida} implies that $S(t)1=1$. This together with 
the positivity of $S(t)$ imply that $S(t)$ is a contraction in $L^\infty(m)$. Finally, the classical theory of interpolation between the Banach spaces $L^1(m)$ and $L^\infty(m)$  implies that $S(t)$ is a contraction in $L^p(m)$ for all $1\leq p\leq\infty$, see Triebel \cite{Triebel}.
\endproof

We have the following proposition which can be applied to functions whose derivatives of orders 1 and 2 are in $L^2(m)$. 

\begin{proposition} \label{prop_subharm}
Let $(X,\Lc)$, $g$ and $m$ be as in Proposition \ref{prop_delta_max}. Then  every  function  $u_0$ in 
$\Dom(\Delta)$ (which is equal to $\Dom(\widetilde\Delta)$) such that $\Delta u_0\geq  0$  (resp.  $\widetilde{\Delta}u_0\geq  0$)  is  constant on the leaf $L_a$ for $m$-almost every $a$. Moreover, if $m$ is an extremal positive harmonic measure, then $u_0$ is constant $m$-almost everywhere.
\end{proposition}
\proof
We know that 
$$\int \Delta u_0 dm =\int \widetilde\Delta u_0 dm=0.$$
So, the hypothesis implies that $\Delta u_0=0$ (resp. $\widetilde\Delta u_0=0$). 
By Lemmas \ref{lemma_e_tilde} and \ref{lemma_e_e_tilde}, we deduce that
$$\int |\nabla u_0|^2dm=-\int(\widetilde\Delta u_0)u_0dm =-\int (\Delta u_0)u_0dm=0.$$
It follows that $\nabla u_0=0$ almost everywhere with respect to $m$. 
Thus, $u_0$ is constant on $L_a$ for $m$-almost every $a$. 
When $m$ is extremal, this property implies that $u_0$ is constant $m$-almost everywhere, since every measurable set of leaves has zero or full $m$ measure.
\endproof

We deduce from the above results the following ergodic theorem.

\begin{corollary} \label{cor_ergodic}
Under the hypothesis of Theorem \ref{th_heat_real}, for all $u_0\in  L^p(m),$  $1\leq p<\infty$, the average 
$$\frac{1}{R}\int_0^R S(t)u_0 dt$$ 
converges pointwise  $m$-almost everywhere
 and also in $L^p(m)$ to 
an $S(t)$-invariant function  $u_0^*$ when $R$ goes to infinity.
Moreover, $u_0^*$ is constant on the leaf $L_a$ for $m$-almost every $a$.
If $m$ is an extremal harmonic measure, then $u$ is constant $m$-almost everywhere.
\end{corollary}
\proof
The first assertion is a consequence of the ergodic theorem as in Dunford-Schwartz
\cite[Th. VIII.7.5]{DunfordSchwartz}. We get a   function $u_0^*$  which is  $S(t)$-invariant.
For the rest of the proposition, since $S(t)$ is a contraction 
in $L^p(m)$, it is enough to consider the case where 
$u_0$ is in $\Dc^0(X)$.

By Proposition \ref{prop_subharm}, we only have to check that $u_0^*$
is 
in the domain of $A$ and $Au_0^*=0$. 
Define
$$u_R:=\frac{1}{R}\int_0^R S(t)u_0 dt.$$ 
This function belongs to $\Dom(A)$. 
Since $u_R$ converges to $u_0^*$ in $L^2(m)$ and the graph of $A$ is closed in $L^2(m)\times L^2(m)$, it is enough to show
that 
$Au_R\rightarrow 0$ in $L^2(m)$. We have
$$Au_R=\frac{1}{R}\int_0^RAu(t,\cdot)dt=-\frac{1}{R}\int_0^R{\partial\over \partial t}u(t,\cdot)dt
={1\over R}u_0-{1\over R}u(R,\cdot).$$
Since $S(t)$ is a contraction in $L^2(m)$, the last expression tends to 0 in $L^2(m)$. The result follows. 
\endproof

We will need the following lemma.

\begin{lemma} \label{lemma_harm_pos}
Let $\widehat m=\theta m$ be a harmonic measure, not necessarily positive, where $\theta$ is a function in $L^2(m)$. Let $\widehat m=\widehat m^+ - \widehat m^-$ be the minimal decomposition 
of $\widehat m$ as the difference of two positive measures. Then $\widehat m^\pm$ are harmonic.
\end{lemma}
\proof
Let $S(t)$ be the semi-group of contractions in $L^1(m)$ associated with $-\Delta$ as above. Define the action of $S(t)$ on measures by
$$\langle S(t)\widehat m,u_0\rangle:=\langle \widehat m, S(t)u_0\rangle\quad \mbox{for}\quad u_0\in L^2(m).$$ 
Consider a function $u_0\in \Dc^0(X)$ and define $\eta(t):=\langle S(t)\widehat m,u_0\rangle$. By Theorem \ref{th_hille_yosida}, this is a $\Cc^1$ function on $\R_+$. We have since $\widehat m$ is harmonic and $\theta$ is in $L^2(m)$
$$\eta'(t)=\langle \widehat m, S'(t)u_0\rangle=\langle \widehat m, -\Delta (S(t)u_0)\rangle =0.$$
To see the last equality, we can use a partition of unity and the local description of $\widehat m$ on a flow box. 
So, $\eta$ is constant. It follows that $S(t)\widehat m=\widehat
m$. Since $S(t)$ is a positive  
contraction, we deduce that $S(t)\widehat m^\pm = \widehat m^\pm$. So, the functions $\eta^\pm(t):=\langle m^\pm,S(t)u_0\rangle$ are constant.
As above, we have
$$\langle \widehat m^\pm,\Delta u_0\rangle = -\langle \widehat m^\pm, S'(0)u_0\rangle=(\eta^\pm)'(0)=0.$$
Hence, $\widehat m^\pm$ are harmonic.
\endproof

We also obtain the following result, see Candel-Conlon \cite{CandelConlon2}.

\begin{corollary} \label{cor_choquet}
Under the hypothesis of Proposition \ref{prop_delta_max}, the family $\Hc$ of harmonic probability measures on $X$ is a
non-empty compact simplex and for any $m\in\Hc$ there 
is a unique probability measure $\nu$ on the set of extremal elements in $\Hc$ such that 
$m=\int m' d\nu(m')$. Moreover, two different harmonic probability measures are mutually singular.
\end{corollary}
\proof
Elements in $\Hc$ are defined as probability measures $m$ such that $\langle m,\Delta u\rangle=0$ for $u\in\Dc^0(X)$. It is clear that $\Hc$ is convex and compact. The fact that $\Hc$ is non-empty is well-known and is a consequence of 
the Hahn-Banach theorem, see L. Garnett \cite{Garnett}. Indeed, by maximum principle for subharmonic functions,
the distance in $\Dc^0_0(X)$ between $1$ and the space $\{\Delta\varphi,\
\varphi\in\Dc^0(X)\}$ is equal to 1. Therefore, Hahn-Banach's theorem
implies the existence of a linear form $m$ on $\Dc^0_0(X)$ of norm $\leq 1$
such that $m(1)=1$. The form defines a harmonic measure of mass $\leq
1$ which is necessarily a probability measure because $m(1)=1$.

By Choquet's representation theorem \cite{Choquet}, we can decompose $m$ into extremal measures as in the corollary. We show that the decomposition is unique. Define $m:=m_1+m_2$ and $\theta_i$ a function in $L^1(m)$, $0\leq \theta_i\leq 1$, such that $m_i=\theta_im$. 
Define also $m_1\vee m_2:=\max\{\theta_1,\theta_2\}m$ and $m_1\wedge m_2:=\min\{\theta_1,\theta_2\}m$.
As in Proposition \ref{prop_choquet_c}, it is enough to show that these measures belong to the cone generated by $\Hc$. By Lemma \ref{lemma_harm_pos}, since $m_1-m_2$ is harmonic, $m':=\max\{\theta_1-\theta_2,0\}m$ is harmonic. It follows that $m_1\vee m_2=m'+m_2$ and $m_1\wedge m_2=m_1-m'$ are harmonic. This completes the proof of the first assertion as a consequence of the Choquet-Meyer theorem \cite[p.163]{Choquet}.

Consider now two different 
extremal elements $m,m'$ in $\Hc$. We show that they are mutually singular.
Consider the decomposition of $m$ and $m'$ in a flow box $\U\simeq \B\times \T$ as above
$$m=\int h_a [\B\times\{a\}]\wedge \Omega d\mu(a) \quad \mbox{and}\quad
m'=\int h_a' [\B\times\{a\}]\wedge \Omega d\mu'(a).$$
We can restrict $\mu,\mu'$ in order to assume that $h_a\not=0$ for $\mu$-almost every $a$ and 
$h_a'\not=0$ for $\mu'$-almost every $a$.
Write $\mu'=\varphi \mu+\mu''$ where $\varphi$ is a positive function in $L^1(\mu)$ and $\mu,\mu''$ are mutually singular.
Choose a measurable set $A$ on the transversal $\T$ such that $\mu''(\T\setminus A)=0$ and $\mu(A)=0$. 
Let $\Sigma$ denote the union of the leaves which intersect $A$. Then, $m$ has no mass on $\Sigma$.  
If $\mu''$ is non-zero, then $m'$ has positive mass on $\Sigma$. Since $m'$ is extremal, it has no mass outside $\Sigma$. We deduce that $m$ and $m'$ are mutually singular.

Assume now that $\mu''=0$. 
Multiplying $h_a$ by $\varphi(a)$ allows us to assume that $\mu=\mu'$. Therefore, there is a non-negative function 
$\theta\in L^1(m)$ such that $m'=\theta m$. 
We can find a number $c>0$ such that $\{\theta\geq c\}$ and $\{\theta\leq c\}$ have positive $m$-measure.
Define $m_c^+:=\max\{\theta-c,0\}m$ and $m_c^-:=\max\{c-\theta,0\}m$. By Lemma 
\ref{lemma_harm_pos}, since 
$m_c^+=(m'\vee cm)-cm$ and $m_c^-=(cm\vee m')-m'$, these measures are harmonic.
So, we can choose a set $\Sigma'$ which is a union of leaves such that  $m_c^+$ has no mass outside $\Sigma'$ and $\theta> c$  on $\Sigma'$. The choice of $c$ implies that
 $m$  has positive mass outside $\Sigma'$. Since $m$ is extremal, we deduce that $m$ has no mass on $\Sigma'$.
It follows that $m_c^+=0$ and then 
$\theta\leq c$ almost everywhere. Using $m_c^-$, we prove in the same way that $\theta\geq c$ almost everywhere. Finally\footnote{We didn't find a simpler argument. The fact that $q(u,v)$ is not symmetric is a difficulty.}, we have $m'=cm$ and since $m,m'$ have the same mass, we get $m=m'$. This is a contradiction and completes the proof.
\endproof

The following result gives us a version of mixing property in our
context. 
The classical case is due to 
Kaimanovich \cite{Kaimanovich} who uses in particular the smoothness
of the Brownian motion, see also Candel \cite{Candel2} who relies on a
version of the zero-two law due to Ornstein and Sucheston \cite{Ornstein}.

\begin{theorem} \label{th_mixing}
Under the hypothesis of Theorem \ref{th_heat_real}, assume moreover that $m$ is extremal. If $S(t)$ is associated to $-\widetilde\Delta$,  then
$S(t)u_0$ converge to $\langle m,u_0\rangle$ in $L^p(m)$ when $t\rightarrow\infty$ 
for $u_0\in L^p(m)$ with $1\leq p<\infty$. In particular, $S(t)$ is mixing, i.e.
$$\lim_{t\rightarrow\infty} \langle S(t)u_0,v_0\rangle=\langle m,u_0\rangle \langle m,v_0\rangle\quad \mbox{for}\quad u_0,v_0\in L^2(m).$$
\end{theorem}
\proof
Since $S(t)$ is a contraction in $L^p(m)$ for $1\leq p\leq\infty$, it is enough to show the above convergence 
for $u_0$ in a dense subspace of $L^p(m)$. If $u_0$ is bounded, we have, $\|S(t)u_0\|_\infty\leq \|u_0\|_\infty$. Therefore, we only have to show the above convergence in $L^1(m)$. 
Since $S(t)$ preserves constant functions, we can assume without loss of generality that $\langle m,u_0\rangle=0$. We can also assume that $u_0\in H^1(m)$ and we will 
show that $S(t)u_0\to 0$ in $H^1(m)$.
 
Define $u:=S(t)u_0$. Using Cauchy-Schwarz's inequality,  Lemma \ref{lemma_e_e_tilde} and the last assertion in Theorem \ref{th_hille_yosida}, we obtain that 
$$\|\nabla u\|_{L^2(m)}^2
\leq \|\widetilde\Delta u\|_{L^2(m)}\|u\|_{L^2(m)}\lesssim{1\over t}\|u_0\|^2_{L^2(m)}.$$
It follows that $S(t)$ is bounded in $H^1(m)$ uniformly on $t$. 
So, we can consider a limit value $u_\infty$ of $S(t)u_0$, $t\to\infty$, in the weak sense in $H^1(m)$. 
We have
$$\|\nabla u_\infty\|_{L^2(m)}\leq\limsup_{t\rightarrow\infty}\|\nabla u\|_{L^2(m)}=0.$$
Using the description of $m$ and $u_\infty$ in flow boxes, we deduce that $u_\infty$ is constant on almost every leaf with respect to $m$. Since $m$ is extremal, $u_\infty$ is constant $m$-almost everywhere.
Finally, since $S(t)$ preserves $m$, we have  $\langle m,S(t)u_0\rangle=0$ and then $\langle m,u_\infty\rangle=0$. It follows that $u_\infty=0$. We deduce that $S(t)u_0\to 0$ weakly in $H^1(m)$.

Now, recall that $-\widetilde\Delta$ is self-adjoint on $L^2(m)$. Hence, $S(t)$ is also self-adjoint since it is obtained as the limit of  $\big(I-tn^{-1}\widetilde\Delta)^n$ where $I$ is the identity operator. Therefore, we have
$$\|S(t)u_0\|_{L^2}^2=\langle S_{2t} u_0,u_0\rangle_{L^2(m)}\rightarrow 0.$$
It follows that $S(t)u_0\rightarrow 0$ in $H^1(m)$. Note that
using E.M. Stein's theorem in \cite{Stein} on self-adjoint contractions of $L^p$, we can prove that $S(t)u_0\to \langle m,u_0\rangle$ $m$-almost everywhere for $u_0\in L^1(m)$. 
\endproof

\begin{proposition} \label{prop_dense_laplace}
Let $m=\int m'd\nu(m')$ be  as in Corollary \ref{cor_choquet}. 
Then the closures of $\Delta(\Dc^0(X))$ and of $\widetilde\Delta(\Dc^0(X))$ in $L^p(m)$, 
$1\leq p\leq 2$, is the space of functions 
$u_0\in L^p(m)$ such that $\int u_0dm'=0$ for $\nu$-almost every $m'$.  In particular, if $m$ is an 
extremal harmonic probability measure, then this space is the hyperplane of $L^p(m)$ defined by the equation 
$\int u_0dm=0$.
\end{proposition}
\proof
We only consider the case of $\Delta$; the case of $\widetilde\Delta$ is treated in the same way. 
It is clear that $\Delta(\Dc^0(X))$ is a subset of the space of $u_0\in L^p(m)$ such that $\int u_0dm'=0$ for $\nu$-almost every $m'$ and the last space is closed in $L^p(m)$. Consider a function $\theta\in L^q(m)$, with $1/p+1/q=1$, which is orthogonal to $\Delta(\Dc^0(X))$. So, $\theta m$ is a harmonic measure.
Since $p\leq 2$, we have $\theta\in L^2(m)$. 
We have to show that $\theta$ is constant with respect to $\nu$-almost every $m'$. 

Consider the disintegration of $m$ along the fibers of $\theta$. There is a probability measure $\nu'$ 
on $\R$ and probability measures $m_c$ on $\{\theta=c\}$ such that $m=\int m_c d\nu'(c)$.  
By Lemma \ref{lemma_harm_pos}, for any $c\in\R$, the measure $\max\{\theta,c\}m$ is harmonic. Therefore, $m_c$ is harmonic for $\nu'$-almost every $c$.  If $\nu_c$ is the probability measure associated with $m_c$ as in Corollary \ref{cor_choquet}, we deduce from the uniqueness in this lemma that 
$$\nu=\int\nu_c d\nu'(c).$$
Now, since $\theta$ is constant $m_c$-almost everywhere, it is constant with respect to $\nu_c$-almost every $m'$. We deduce from the above identity that $\theta$ is constant for $\nu$-almost every $m'$. This completes the proof.
\endproof

\begin{remark}\rm
We will see in the next section an analogous development for the heat equation on
Riemann  surface laminations with isolated singularities. It is 
known that a compact Riemann surface lamination with a tame set of singularities always admits 
a harmonic probability measure  \cite{BerndtssonSibony}. 
For Riemannian laminations it would be interesting to find natural hypothesis on singularities  
which guarantees the existence of such a measure.
\end{remark}

\section{Case of singular Riemann surface laminations} \label{section_heat_complex}

We will consider in this section the heat equation for positive harmonic 
currents associated with a compact 
Riemann surface lamination possibly with singularities. The main
example we have in mind is the case of a current supported on the set
of hyperbolic leaves which are endowed with the Poincar{\'e} metric. Since
in general, Poincar{\'e}'s metric does not depend continuously on the
leaves, it is important that we can relax the strong regularity of the
metric which is a necessary condition in the real setting. 

Consider a more general situation. Let $T$ be a positive $\ddbar$-closed
current of bidimension $(1,1)$ with compact support in a complex manifold $M$ of dimension
$k$. For simplicity, fix a Hermitian form $\omega$ on $M$. So,
$T\wedge\omega$ is a positive measure. 
We assume that $T$ is regular in the sense of
\cite{BerndtssonSibony}, that is, there exists a $(1,0)$-form $\tau$
defined almost everywhere with respect to $T\wedge\omega$ such that
$$\partial T=\tau\wedge T \quad \mbox{and} \quad \int i\tau\wedge
\overline{\tau}\wedge T<\infty.$$ 
In this context, the H{\"o}rmander $L^2$-estimates are proved in
\cite{BerndtssonSibony} for the $\dbar$-equation induced on $T$. 
Fix also a positive $(1,1)$-form $\beta$ which is defined
almost everywhere with respect to $T\wedge\omega$ such that
$T\wedge\beta$ is of finite mass. We assume that
$\beta$ is {\it strictly $T$-positive} in the sense that $T\wedge\omega$ is
absolutely continuous with respect to $T\wedge\beta$. 
This condition\footnote{The form $\beta$ plays the role of a ``Hermitian
  metric" on the current $T$ that can be seen as a ``generalized
  submanifold" of $M$.}
does not depend on the choice of 
$\omega$ and allows us to define 
the operators $\nabla^\partial_\beta$,
$\nabla^\dbar_\beta$ and $\nabla_\beta$ on $u\in\Dc^0(M)$ by
$$(\nabla^\partial_\beta u) T\wedge \beta:=i\partial u\wedge 
\overline \tau\wedge T=i\partial(u\dbar T),\quad
(\nabla^\dbar_\beta u) T\wedge \beta:=-i\dbar u\wedge \tau\wedge T=-i\dbar(u\partial T)$$
and 
$$\nabla_\beta:=\nabla^\partial_\beta+\nabla^\dbar_\beta.$$
Define also the
operators $\Delta_\beta$ and $\widetilde\Delta_\beta$ on
$u\in\Dc^0(M)$ by
$$(\Delta_\beta u) T\wedge\beta:=i\ddbar u\wedge T 
\quad \mbox{and}\quad \widetilde\Delta_\beta:=\Delta_\beta+{1\over 2}\nabla_\beta.$$

Let $m_\beta$ denote the measure $T\wedge\beta$. 
We  introduce  the Hilbert space  $H^1_\beta(T)\subset L^2(m_\beta)$ 
associated with $T$ and  $\beta$  as the  completion of $\Dc^0(M)$ with  respect  
to the norm\footnote{The second integral does not depend on $\beta$.
This is an important difference in comparison with the analogous notion in the real setting.}
$$\|u\|^2_{H^1_\beta}:=\int |u|^2  T\wedge\beta +i\int \partial u
\wedge \dbar u\wedge T.$$
Observe 
the operators $\nabla^\partial_\beta$, $\nabla^\dbar_\beta$ and $\nabla_\beta$ are defined on $H^1_\beta(T)$ with values in $L^1(m_\beta)$. 

Define also for $u,v\in\Dc^0(M)$ (for simplicity, we only consider real-valued functions)
$$q(u,v):=-\int (\Delta_\beta u)vT\wedge\beta, \quad e(u,v):=q(u,v)+\int uv T\wedge\beta$$
and
$$\widetilde q(u,v):=-\int (\widetilde \Delta_\beta u)vT\wedge\beta, \quad 
\widetilde e(u,v):=\widetilde q(u,v)+\int uv T\wedge \beta.$$

We will define later the
domain of $\Delta$ and $\widetilde\Delta$ which allows us to extend
these identities to more general $u$ and $v$. 
The following lemma also holds for $u,v$ in 
the domain of $\Delta$ and $\widetilde\Delta$.

\begin{lemma} \label{lemma_e_tilde_c}
We have 
for $u,v\in\Dc^0(M)$
$$\widetilde q(u,v)=\Re\int  i\partial u\wedge \dbar v \wedge T 
\quad \mbox{and}\quad \int (\widetilde \Delta_\beta u)v T\wedge\beta = \int
u(\widetilde \Delta_\beta v) T\wedge\beta.$$
In particular, $\widetilde q(u,v)$ and $\widetilde e(u,v)$ are symmetric in $u,v$ and 
$$\int (\widetilde\Delta_\beta u) T\wedge\beta=
\int (\Delta_\beta u) T\wedge\beta=\int (\nabla_\beta u) T\wedge\beta 
=0 \quad \mbox{for}\quad u\in\Dc^0(X).$$
\end{lemma}
\proof
Note that $i\ddbar u$ is a real form.
Since $T$ is $\ddbar$-closed, the integral of $i\ddbar u^2\wedge T$
vanishes. We deduce using Stoke's formula that
\begin{eqnarray*}
\widetilde q(u,v) & = & -\int (\Delta_\beta u+ {1\over 2}\nabla_\beta u)  v T\wedge\beta\\
& = &  -\int  i\ddbar u\wedge v T
-\Re\int i \partial u\wedge \overline \tau \wedge vT\\
& = &  -\Re\int  i\ddbar u\wedge v T
-\Re\int i \partial u\wedge \overline \tau \wedge vT\\
& = &  \Re\int i\partial u\wedge \big[\dbar (v T)-v\dbar T\big]=
\Re\int i\partial u\wedge \dbar v\wedge T.
\end{eqnarray*}
This gives the first identity in the lemma.

We also have since $T$ is $\ddbar$-closed
$$\int(\nabla_\beta u) T\wedge\beta=2\Re\int i\partial u\wedge \dbar T=2\Re\int -iu\wedge \ddbar T=0.$$
The other assertions in the lemma are obtained 
as in Lemma \ref{lemma_e_tilde}.
\endproof

We say that $T$ is {\it $\beta$-regular} if $i\tau\wedge\overline\tau\wedge T
\leq \beta\wedge T$. We will see that this hypothesis is satisfied for foliations with linearizable singularities and for $\beta:=\omega_P$. 
We have the following lemma.

\begin{lemma} \label{lemma_e_e_tilde_c}
The bilinear forms $\widetilde q$ and $\widetilde e$ extend
  continuously to $H^1_\beta(T)\times H^1_\beta(T)$. 
Assume that $T$ is $\beta$-regular. Then the same property holds for $q$ and $e$.
Moreover, we have
$q(u,u)=\widetilde q(u,u)$ and $e(u,u)=\widetilde e (u,u)$ for $u\in H^1_\beta(T)$. 
\end{lemma}
\proof
The first assertion is deduced from Lemma
\ref{lemma_e_tilde_c}. Assume that   $T$ is $\beta$-regular.
Using Cauchy-Schwarz's inequality, we have for $u,v\in\Dc^0(M)$
\begin{eqnarray*}
|q(u,v)-\widetilde q(u,v)|^2 & \leq & \Big|\int \partial u \wedge
v\overline\tau \wedge T \Big|^2 \\
& \leq &  \Big(\int i\partial u\wedge \dbar u
\wedge T \Big) \Big(\int iv^2\tau\wedge\overline\tau \wedge T\Big)\\  
& \leq & \Big(\int i\partial u\wedge \dbar u
\wedge T \Big) \Big(\int v^2 \beta \wedge T\Big)
\end{eqnarray*}
which implies the second assertion. We also have for $u\in\Dc^0(M)$
\begin{eqnarray*}
q(u,u)-\widetilde q(u,u) & = & \Re\int(\nabla_\beta^\partial
u)uT\wedge\beta=\Re\int i\partial u\wedge u \overline\tau\wedge T \\
& = & {1\over 2}\Re\int i\partial u^2 \wedge \dbar T=-{1\over 2}\Re\int
iu^2 \ddbar T= 0. 
\end{eqnarray*}
This allows us to obtain the other properties as in Lemma \ref{lemma_e_e_tilde_c}.
\endproof

Define the domain 
$\Dom(\widetilde\Delta_\beta)$ of $\widetilde\Delta_\beta$
(resp. $\Dom(\Delta_\beta)$ of $\Delta_\beta$ when $T$ is $\beta$-regular) as the space of $u\in
H^1_\beta(T)$ such that $\widetilde q(u,\cdot)$ (resp. $q(u,\cdot)$) extends to a linear
continuous form on $L^2(m_\beta)$. When $T$ is $\beta$-regular, we
have seen in the proof of Lemma \ref{lemma_e_e_tilde_c} that
$q(u,v)-\widetilde q(u,v)$ is continuous on $H^1_\beta(T)\times
L^2(m_\beta)$. Therefore, $\Dom(\Delta_\beta)=\Dom(\widetilde\Delta_\beta)$.
We have the following proposition.

\begin{proposition} \label{prop_delta_max_c}
Let $T$ be a positive $\ddbar$-closed current of bidimension 
$(1,1)$ on a complex manifold $M$ and $\beta$ 
a strictly $T$-positive form as above. Then the associated operator
$-\widetilde\Delta_\beta$ (resp. $-\Delta_\beta$ when $T$ is
$\beta$-regular) is maximal monotone on $L^2(m_\beta)$ 
where $m_\beta:=T\wedge\beta$. In particular,
it is the infinitesimal generator of semi-groups of contractions on $L^2(m_\beta)$ and its graph is closed.
\end{proposition}
\proof
By Lemmas \ref{lemma_e_tilde_c} and \ref{lemma_e_e_tilde_c}, we have for $u\in H_\beta^1(T)$
$$\widetilde e(u,u)\geq \|u\|^2_{H_\beta^1} \quad \mbox{and}\quad
e(u,u)\geq \|u\|^2_{H_\beta^1}.$$
We only have to follow the same lines as in the proof of Proposition \ref{prop_delta_max}.
\endproof

We have the following theorem.

\begin{theorem} \label{th_heat_c} 
Under the hypothesis of Proposition \ref{prop_delta_max_c}, let $S(t)$,
$t\in\R_+$, denote  the semi-group of contractions associated with the operator
$-\widetilde\Delta$ (resp. $-\Delta$ when $T$ is $\beta$-regular)
which is given by the Hille-Yosida
theorem. Then the  
measure $m_\beta:=T\wedge\beta$ is  $S(t)$-invariant and
$S(t)$ is a positive contraction in $L^p(m_\beta)$ for all $1\leq p\leq\infty$. 
\end{theorem}
\proof
The proof is almost the same as in Theorem \ref{th_heat_real}. We only give here some computations which are 
slightly different. We have with mostly the same notation
$$\xi(t):=\int H(u(t,\cdot)-K)T\wedge\beta$$
and
$$\xi'(t)=-\int  G(v(t,\cdot))Av(t,\cdot) T\wedge \beta$$
with $A=-\widetilde \Delta$ or $A=-\Delta$. 
In the first case, by Lemma \ref{lemma_e_tilde_c}, the last integral is equal to
$$-\Re\int i\partial  G(v)\wedge\dbar v \wedge T =-\Re\int iG'(v)\partial v\wedge\dbar v\wedge T
=-\int iG'(v)\partial v\wedge\dbar v\wedge T\leq 0.$$
When $A=-\Delta_\beta$, there is an extra term which is equal to
\begin{eqnarray*}
-{1\over 2}\int  G(v)\nabla_\beta v T\wedge \beta & = & 
-\Re\int  iG(v)\partial v\wedge\dbar T\\
& = & -\Re\int  i\partial H(v)\wedge\dbar T=\Re\int i H(v)\ddbar T=0.
\end{eqnarray*}
We can now follow closely the proof of Theorem \ref{th_heat_real}.
\endproof

As in the real case, we deduce from the above results the following ergodic theorem.

\begin{corollary} \label{cor_ergodic_c}
Under the hypothesis of Theorem \ref{th_heat_c}, for all 
$u_0\in  L^p(m_\beta),$  $1\leq p<\infty$, the average 
$$\frac{1}{R}\int_0^R S(t)u_0 dt$$ 
converges pointwise  $m_\beta$-almost everywhere
 and also in $L^p(m_\beta)$ to 
an $S(t)$-invariant function  $u_0^*$ when $R$ goes to infinity.
\end{corollary}

Consider now a compact Riemann surface lamination $(X,\Lc,E)$ with isolated singularities in $M$
and $T$ a positive harmonic current on the lamination. 
We have seen in 
Proposition \ref{prop_current_extension} that $T$ 
is also a positive $\ddbar$-closed current of bidimension $(1,1)$ on $M$. 
By Proposition \ref{prop_current_local_c} and Lemma
\ref{prop_poincare_measurable}, 
we can decompose $T$ into the sum of a positive harmonic current on
the union of the parabolic leaves and another on the union of the
hyperbolic leaves. 
We want to study the second part of $T$. From now on, assume for simplicity 
that $T$ has no mass on the union of the parabolic leaves. 

With the local decomposition of $T$, we have
$\partial T=\tau\wedge T$ with
$\tau=h_a^{-1}\partial h_a$ when $h_a\not =0$ and $\tau=0$ otherwise. On almost every leaves, when $\tau\not=0$, 
$i\tau\wedge\overline\tau$ defines a metric with curvature $-1$. It
follows that $i\tau\wedge\overline\tau$ is bounded by the Poincar{\'e}
metric $\omega_P$ on the leaves, see \cite{FornaessSibony3} for details. 
Therefore, all the above results can be applied for $T$ with 
$\beta:=\omega_P$ if $T\wedge\omega_P$ is a measure of finite mass, 
in particular, for laminations with linearizable singularities, e.g. generic
foliations in $\P^k$, 
see Proposition \ref{prop_poincare_mass} (they can be also applied to $i\tau\wedge\overline\tau$ and $\widetilde\Delta$ when $\tau\not=0$). 
In what follows, we will use the index $P$ instead of $\beta$,
e.g. $\Delta_P$ instead of $\Delta_\beta$.
Note that $T\wedge\omega_P$ is a harmonic measure on $X$ with respect to the Poincar\'e metric on the leaves.
We have the following result.

\begin{proposition} \label{prop_subharm_c_bis}
Let $(X,\Lc,E)$ be a compact Riemann surface lamination in a complex 
manifold $M$ with isolated singularities and $\omega_P$ the 
Poincar{\'e} metric on the leaves. Let $T$ be a positive harmonic 
current on the lamination without mass on the set of parabolic
leaves. 
Assume that $m_P:=T\wedge \omega_P$ is a measure of finite mass. Then  every  function  $u_0$ in 
$\Dom(\Delta_P)$ such that $\Delta_P u_0\geq  0$  
(resp.  $\widetilde \Delta_P u_0\geq  0$)  
is  constant on the leaf $L_a$ for $m_P$-almost every $a$. If moreover $T$ is an extremal positive 
harmonic current, then $u_0$ is constant $m_P$-almost everywhere.
\end{proposition}
\proof
We have seen in the above discussion that $T$ is $\omega_P$-regular.
So, using Lemmas \ref{lemma_e_tilde_c} and \ref{lemma_e_e_tilde_c} we have the identities
$$\int \Delta_P u_0 T\wedge\omega_P =\int \widetilde\Delta_P u_0 T\wedge\omega_P=0$$
and
$$\int i\partial u_0\wedge\dbar u_0 \wedge T
=-\int(\widetilde\Delta_Pu_0)u_0 T\wedge\omega_P =-\int(\Delta_Pu_0)u_0 T\wedge\omega_P=0.$$
So, we can repeat the proof of Proposition \ref{prop_subharm}. 
\endproof

Note that since the Poincar{\'e} metric is not bounded in general, we
don't have a priori a one to one correspondence 
between harmonic currents and harmonic measures. However, we have the following lemma that can be easily obtained using local description of currents.

\begin{lemma} \label{lemma_corr_c}
Let $T$ be a positive harmonic current with finite Poincar{\'e}'s mass.
If $m$ is a positive harmonic measure such that $m\leq m_P:=T\wedge\omega_P$, then there is a positive harmonic current $S\leq T$ such that $m=S\wedge\omega_P$. In particular, if $T$ is extremal, then 
$m_P$ is an extremal positive harmonic measure.
\end{lemma}

This allows us to prove the following result as in Corollary \ref{cor_ergodic}. 

\begin{corollary} \label{cor_ergodic_c_bis}
Under the hypothesis of Proposition \ref{prop_subharm_c_bis} 
and with the notation of Corollary \ref{cor_ergodic_c},
$u_0^*$ is constant on the leaf $L_a$ for $m_P$-almost every $a$. Moreover, if 
$T$ is an extremal positive harmonic current, then $u_0^*$ is constant $m_P$-almost everywhere.
\end{corollary}

The following result gives us the mixing for the operator $-\widetilde\Delta_P$. We can also obtain a pointwise
convergence using E.M. Stein's theorem.

\begin{theorem} 
Under the hypothesis of Proposition \ref{prop_subharm_c_bis}, if $S(t)$ is associated 
with $-\widetilde\Delta_P$ and if $T$ is extremal,  
$S(t)u_0\to \langle m_P,u_0\rangle$ in $L^p(m)$ when $t\to \infty$
for every $u_0\in L^p(m_P)$ with $1\leq p<\infty$.
\end{theorem}
\proof
We follow closely the proof of Theorem \ref{th_mixing} using similar notation. In this case, we use flow boxes away from the singularities of the lamination.
We only need to 
notice that 
$$2i\int \partial u\wedge\dbar u\wedge T=\|\nabla u\|_{L^2(T\wedge\omega)}^2$$
where $\nabla$ is the gradient along the leaves with respect to the metric $\omega$.
\endproof

As in the real case, we have the following result.

\begin{proposition} \label{prop_dense_laplace_c}
Let $T=\int T'd\nu(T')$ be  as in Propositions \ref{prop_subharm_c_bis} and \ref{prop_choquet_c}. 
Then the closures of $\Delta_P(\Dc^0(X))$ and of $\widetilde\Delta_P(\Dc^0(X))$ in $L^p(m_P)$, 
$1\leq p\leq 2$, is the space of functions 
$u_0\in L^p(m_P)$ such that $\int u_0 T'\wedge\omega_P =0$ for $\nu$-almost every $T'$.  
In particular, if $T$ is an 
extremal positive harmonic current of finite Poincar{\'e}'s mass, then this space is the 
hyperplane of $L^p(m_P)$ defined by the equation $\int u_0 T\wedge\omega_P=0$.
\end{proposition}
\proof
The proof follows the lines as in Proposition \ref{prop_dense_laplace}. With the analogous notation, we obtain that $\theta m$ restricted to $\{c\leq\theta\leq c'\}$ is harmonic for all $c,c'$. By Lemma \ref{lemma_corr_c}, this measure is associated to a harmonic current. This allows us to disintegrate $T$ along the fibers of $\theta$ and to conclude using Proposition \ref{prop_choquet_c}. 
\endproof

\begin{remark}\rm
For all the above results, when the lamination has no singular points, 
the hypothesis that it is embedded in a complex manifold is
unnecessary. 
In general, it is enough to assume that locally near singular points,
the 
lamination can be embedded in a complex manifold.
\end{remark}

\section{Birkhoff's type  theorem for laminations} \label{section_birkhoff}

In this section, we will give an analogue  of  
Birkhoff's ergodic theorem in the  context of a compact lamination $(X,\Lc,E)$
with isolated singularities on a complex manifold $M$. In the case where the
set of singularities is empty, the property that the lamination is
embedded in a complex manifold is unnecessary. Let $T$ be a positive
harmonic current on $(X,\Lc,E)$. We have seen that $T$ is also a
positive $\ddbar$-closed current of bidimension $(1,1)$ on $M$. We
assume that $T$ has no mass on the union of parabolic leaves and that
$m_P:=T\wedge\omega_P$ is a probability measure where $\omega_P$ denotes
the Poincar{\'e} metric on the leaves. So, $m_P$ is a harmonic measure on $X$ with respect to $\omega_P$.

For any point $a\in X\setminus E$ such that the corresponding leaf
$L_a$ is hyperbolic, consider a universal covering map
$\phi_a:\D\rightarrow L_a$ such that $\phi_a(0)=a$. This map is
uniquely defined by $a$ up to a rotation on $\D$. Denote by
$r\D$ the disc of center 0 and of radius $r$ with $0<r<1$. In the
Poincar{\'e} metric, this is also the disc of center 0 and of radius 
$$R:=\log{1+r\over 1-r}\cdot$$
So, we will also denote by $\D_R$ this disc.
For all $0<R<\infty$, consider 
$$m_{a,R}:=\frac{1}{M_R}(\phi_a)_* \big(\log^+ \frac{r}{|\zeta|}\omega_P\big).$$
where $\log^+:=\max\{\log,0\}$, $\omega_P$ denotes also the Poincar{\'e} metric on $\D$ and 
$$M_R:= \int \log^+ \frac{r}{|\zeta|}\omega_P=\int \log^+ \frac{r}{|\zeta|} \frac{2}{(1-|\zeta|^2)^2}
id\zeta\wedge d\overline\zeta.$$
So, $m_{a,R}$ is a probability measure  which depends on $a,R$ but 
does not depend on the choice of $\phi_a$. 
Here is the main theorem in this section.

\begin{theorem}\label{th_birkhoff}
Let $(X,\Lc,E)$ be a compact lamination with isolated singularities in a complex manifold $M$ and
$\omega_P$ the Poincar{\'e} metric on the leaves. 
Let $T$ be an extremal positive harmonic current of Poincar{\'e} mass
$1$ on $(X,\Lc,E)$ without mass on the union of parabolic leaves.  
Then for almost every point $a\in X$ with respect to the measure
$m_P:=T\wedge\omega_P$, the measure $m_{a,R}$ defined above converges to $m_P$
when $R\to\infty$.  
\end{theorem}

For  every  $0<R<\infty$, we introduce  the operator $B_R$ by 
$$B_Ru(a):=\frac{1}{M_R}\int_{|\zeta|<1} \log^+
\frac{r}{|\zeta|}(\phi_a)^* (u\omega_P)=\langle
m_{a,R},u\rangle.$$
Note that for $u\in L^1(m_P)$, the function $B_Ru$ is defined
$m_P$-almost everywhere. So, the convergence in 
Theorem \ref{th_birkhoff} is equivalent to the convergence
$B_Ru(a)\rightarrow \langle m_P,u\rangle$ for $u$ continuous and for $m_P$-almost every $a$.

\begin{proposition}\label{prop_B_L1} 
Under the hypothesis of Theorem \ref{th_birkhoff}, 
for  every $u\in L^1(m_P)$, we have 
$$\int (B_R u) T\wedge\omega_P =\int u T\wedge\omega_P.$$
In particular, $B_R$ is positive and of norm $1$ in $L^p(m_P)$ for all
$1\leq p\leq \infty$.  
\end{proposition}
\proof
Fix an $R>0$.
The positivity of $B_R$ is clear. Since $B_R$ preserves constant functions, 
its norm in $L^p(m_P)$ is at least equal to 1.
It is also clear that $B_R$ is an operator of norm 1 on $L^\infty(m_P)$. 
If $B_R$ is of norm 1 on $L^1(m_P)$, by interpolation \cite{Triebel}, its norm on $L^p(m_P)$ is also equal to 1. 
So, the second assertion is a consequence of the first one. Define
another operator $B'_R$ by
$$B'_Ru(a):=\frac{1}{M'_R}\int_{\D_R} (\phi_a)^*
(u\omega_P) \quad \mbox{where}\quad M'_R:=\int_{\D_R} (\phi_a)^*(\omega_P).$$
Note that $M_R'$ is the Poincar{\'e} area of $\D_R$ which is also the Poincar\'e area of $\phi_a(\D_R)$ counted with multiplicity. The operator
$B_R$ can be obtained as an average of $B'_t$ on $t\leq R$. So, it is
enough to prove the first assertion of the proposition for $B'_R$ instead of $B_R$.

We can assume that $u$ is positive and using a partition of unity, 
we can also assume that $u$ has support in a compact set of a flow box
$\U\simeq \B\times\T$ 
as in Section \ref{section_riemann}. Let $\T_1$ be the set of $a\in\T$
such that $L_a$ 
is hyperbolic. We will use the decomposition of $T$ and the notation
as in Proposition 
\ref{prop_current_local_c}. By hypothesis, we can assume that the measure $\mu$ has total
mass on $\T_1$. 
Now, we apply Corollary \ref{cor_lusin} to $\mu$ in stead of $\nu$ and 
$4\lambda R$ instead of $R$ for a fixed constant $\lambda$ large enough. 
Let $\Sigma_n(R)$ denote the union of $L_{a,R}$
for $a\in \S_n$. 
Define by induction 
the function $u_n$ as follows: $u_1$ is the restriction of $u$ to
$\Sigma_1(R)$ 
and $u_n$ is the restriction of
$u-u_1-\cdots-u_{n-1}$ to $\Sigma_n(R)$. We have $u=\sum u_n$. So, it
is enough 
to prove the proposition for each $u_n$. 

We use now the properties of $\S_n$ given in Corollary
\ref{cor_lusin}. 
The set $\Sigma_n(4\lambda R)$ is a smooth lamination and the restriction
$T_n$ 
of $T$ to $\Sigma_n(4\lambda R)$ is a positive harmonic current. 
Observe that $B_Ru_n$ vanishes outside $\Sigma_n(\lambda R)$ and
does not depend on the restriction of $T$ to $X\setminus\Sigma_n(2\lambda R)$. 
Since there is a natural projection from $\Sigma_n$ to the transversal $\S_n$, 
the extremal positive harmonic currents on $\Sigma_n$ are supported by
a leaf and defined by a harmonic function. Therefore, we can reduce the
problem 
to the case where $T=h[L_{a,4\lambda R}]$ with $a\in \S_n$ and $h$ is positive
harmonic on $L_{a,4\lambda R}$. 

Define $\widehat u:=u_n\circ \phi_a$, $\widehat h:=h\circ \phi_a$ and
$\widehat{B_R'u}:=(B_R'u)\circ\phi_a$. The function $\widehat h$ is
harmonic on $\D_{4\lambda R}$. 
Choose a 
measurable set $A\subset \D_{2\lambda R}$ such that $\phi_a$ defines a
bijection 
between $A$ and $L_{a,2\lambda R}$. 
Denote also by $\omega_P$  the Poincar{\'e} metric on $\D$ and
$\dist$ the associated distance. We first observe that 
$$\widehat{B_R'u}(0):={1\over M_R'}\int_{\dist(\zeta,0)<R}\widehat u(\zeta)\omega_P(\zeta).$$
If $\eta$ is a point in $\D$ and $\tau:\D\to\D$ is an automorphism such that $\tau(0)=\eta$, then $\phi_a\circ\tau$ is also a covering map of $L_a$ but it sends 0 to $\phi_a(\eta)$. We apply the above formula to this covering map. Since $\tau$ preserves $\omega_P$ and $\dist$, we obtain
$$\widehat{B_R'u}(\eta):={1\over M_R'}\int_{\dist(\zeta,\eta)<R}\widehat u(\zeta)\omega_P(\zeta).$$
Hence, we have to show the following identity
$$\int_A \Big[{1\over M_R'}\int_{\dist(\zeta,\eta)<R} \widehat
u(\zeta)\omega_P(\zeta)\Big]\widehat h(\eta)\omega_P(\eta) =
\int_A\widehat u(\zeta)\widehat h(\zeta)\omega_P(\zeta).$$

Let $W$ denote the set of points $(\zeta,\eta)\in\D^2$ such that
$\eta\in A$ and $\dist(\zeta,\eta)<R$. Let $W'$ denote the symmetric
of $W$ with respect to the diagonal, i.e. the set of $(\zeta,\eta)$
such that $\zeta\in A$ and $\dist(\zeta,\eta)<R$. Since $\widehat h$ is
harmonic, we have
$$\widehat h(\zeta)= {1\over M_R'}\int_{\dist(\zeta,\eta)<R} \widehat
h(\eta)\omega_P(\eta).$$
Therefore, our problem is to show that the integrals of $\Phi:=\widehat
u(\zeta)\widehat h(\eta)\omega_P(\zeta)\wedge\omega_P(\eta)$ on $W$ and
$W'$ are equal.

Consider the map 
$\phi:=(\phi_a,\phi_a)$ from $\D^2$ to $L_a^2$. The 
fundamental group $\Gamma:=\pi_1(L_a)$ can be identified with a group of
automorphisms of $\D$. Since, $\Gamma^2$ acts on
$\D^2$ and preserves the form $\Phi$, our problem is equivalent to
showing that each fiber of $\phi$ has the same number of points in $W$
and in $W'$. 
We only have to consider the fibers of points in $L_{a,2\lambda R}\times L_{a,2\lambda R}$ since $B_R'u_n$ is supported on $L_{a,\lambda R}$. 
Fix a point $(\zeta,\eta)\in A^2$ and consider the
fiber $F$ of $\phi(\zeta,\eta)$. By definition of $A$, the numbers of
points in $F\cap W$ and $F\cap W'$ are respectively equal to
$$\#\big\{ \gamma\in\Gamma,\ \dist(\gamma\cdot\zeta,\eta)<R\big\}
\quad \mbox{and}\quad \#\big\{ \gamma\in\Gamma,\ \dist(\zeta,\gamma\cdot\eta)<R\big\}.$$
Since $\Gamma$ preserves the Poincar{\'e} metric, the first set is
equal to 
$$\{ \gamma\in\Gamma,\ \dist(\zeta,\gamma^{-1}\cdot\eta)<R\big\}.$$
It is now clear that the two numbers are equal. This
completes the proof.
\endproof

We have the following ergodic theorem.

\begin{theorem} \label{th_ergodic_Lp}  
Under the hypothesis of Theorem \ref{th_birkhoff}, 
if $u$ is a function in $L^p(m_P)$, with $1\leq p <\infty$, then
$B_Ru$ converge in $L^p(m_P)$ towards a constant function
$u^*$  when $R\to \infty$. 
\end{theorem}
\begin{proof}
We show that it is enough to consider the case where $p=1$.
By Proposition \ref{prop_B_L1}, it is enough to consider $u$ in a
dense subset of $L^p(m_P)$, e.g. $L^\infty(m_P)$. 
For $u$ bounded, we have $\|B_Ru\|_\infty\leq \|u\|_\infty$. Therefore, if $B_Ru\to u^*$ in $L^1(m_P)$ we have $B_Ru\to u^*$ in $L^p(m_P)$, $1\leq p<\infty$. 

Now, assume that $p=1$. Since $B_R$ preserves constant functions,
by Proposition \ref{prop_dense_laplace_c} applied to $p=1$, it is enough to consider
$u=\Delta_P v$ with $v\in\Dc^0(M)$. We have to show that $B_Ru$
converges to 0. 
Note that since $v$
is in $\Dc^0(M)$, the function $\Delta_P v$ is defined at every point
outside the parabolic leaves by the formula $(\Delta_P v)
\omega_P:=i\ddbar v$ on the leaves. 

Define 
$$T_{a,R}:=\frac{1}{M_R}(\phi_a)_* \big(\log^+ \frac{r}{|\zeta|}\big).$$
We have $m_{a,R}=T_{a,R}\wedge\omega_P$ and
$$B_Ru(a)= B_R(\Delta_Pv)(a) = \langle  T_{a,R},(\Delta_P v) \omega_P \rangle
= \langle  T_{a,R}, i\ddbar v \rangle=\langle i\ddbar T_{a,R},v\rangle.$$
Since $M_R\rightarrow\infty$, it is easy to see that the mass of
$\ddbar T_{a,R}$ tends to 0 uniformly on $a$. So, the last integral
tends to 0 uniformly on $a$. The result follows.\end{proof}

Let $p(x,y,t)$ be the heat kernel on the disc $\D$ with respect to the Poincar{\'e} metric 
$\omega_P$ on $\D$. This is a positive function on $\D^2\times \R^*_+$ 
smooth when $(x,y)$ is outside the diagonal of $\D^2$. It satisfies
$$\int_\D p(x,y,t) \omega_P (y)=1 \quad \mbox{and} \quad 
{1\over 2\pi}\log{\frac{1}{|y|}}  =\int_0^\infty p(0,y,t)dt.$$
The function $p(0,\cdot,t)$ is radial, see e.g. Chavel
\cite[p.246]{Chavel}.
Define the operator $S_t$ by 
$$S_tu(a):= \int_\D p(0,\cdot,t)(u\circ \phi_a)\omega_P.$$
Observe  that  the family  $S_t$ with $t>0$  is  a  semi-group of operators, i.e. $S_{t+t'}u=S_t\circ S_{t'}u$ for  $u\in \Dc^0(M)$.

\begin{lemma} \label{lemma_heat_delta}
The operator $S_t$ extends continuously to an operator of norm $1$ on  $L^p(m_P)$ for 
$1\leq p\leq\infty$. Moreover, there is a constant $c>0$ such that for 
all $\epsilon>0$ and $u\in L^1(m_P)$, 
we have 
$$m_P\big\{\widetilde Su>\epsilon\big\}\leq c\epsilon^{-1}\|u\|_{L^1(m_P)},$$
where the operator $\widetilde S$ is defined by 
$$\widetilde Su(a):=\limsup_{R\to\infty}\Big|\frac{1}{R}\int_0^R S_tu(a)dt\Big|.$$
\end{lemma}
\proof
It is clear that $S_t$ is positive and preserves constant
functions. Its norm on $L^\infty(m_P)$ is equal to 1. Since
$p(0,\cdot,t)$ is a radial function, 
as in Proposition \ref{prop_B_L1}, $S_t$ is an
average of $B_R'$. We then obtain in the same way the first assertion
of the lemma. The second one is a direct consequence of  Lemma VIII.7.11 in Dunford-Schwartz
\cite{DunfordSchwartz}. This lemma says that if $S_t$ is a semi-group
acting on $L^1(m)$ for some probability measure $m$ such that
$\|S_t\|_{L^1}\leq 1$,  $\|S_t\|_{L^\infty}\leq 1$ and $t\mapsto S_tu$
is measurable with respect to the Lebesgue measure on $t$, then 
$$m\big\{\widetilde Su>\epsilon\big\}\leq
c\epsilon^{-1}\|u\|_{L^1}$$
where $\widetilde S$ is defined as above.
\endproof

Consider also  the  operator $\widetilde B$ given by
$$\widetilde B u(a):=\limsup_{R\to\infty}| B_Ru(a)|.$$
We have the following lemma.

\begin{lemma}\label{lemma_limsup}
There is a constant $c>0$ such that for all $\epsilon>0$ and $u\in L^1(m_P)$ we have 
$$m_P\big\{\widetilde Bu>\epsilon \big\} \leq   c\epsilon^{-1} \|u\|_{L^1(m_P)}.$$ 
\end{lemma}
\proof
Since we can write $u=u^+-u^-$ with $\|u\|_{L^1(m_P)}=\|u^+\|_{L^1(m_P)}+\|u^-\|_{L^1(m_P)}$, it is enough to consider $u$ positive with $\|u\|_{L^1(m_P)}\leq 1$. 
Write $u=\sum_{i\geq 0} u_i$ with $u_i$ positive and bounded such that $\|u_i\|_{L^1(m_P)}\leq 4^{-i}$. 
We will show that  $\widetilde Su_i=\widetilde Bu_i$. This, together with 
Lemma \ref{lemma_heat_delta} applied to $u_i$ and to $2^{-i-1}\epsilon$ give the result.

So, in what follows, assume that $0\leq u\leq 1$. We show that $\widetilde Su=\widetilde Bu$. 
This assertion will be  an immediate consequence  of  the  following  estimate
$$\Big|B_Ru(a)-\frac{2\pi}{M_R}\int_0^{{M_R\over 2\pi}} S_tu(a)dt\Big|\leq  c R^{-1/2}\sqrt{\log R}$$
for $m_P$-almost every $a$ and $c$ a constant independent of $u$.
Observe that the integrals in the left hand side of the last inequality can be computed on $\D$ in term of $\widehat u:=u\circ \phi_a$ and the Poincar{\'e} metric $\omega_P$ on $\D$. So, in order to simplify the notation, we will work on $\D$. We have to show that
$$\Big|B_R\widehat u(0)-\frac{2\pi}{M_R}\int_0^{M_R\over 2\pi} S_t\widehat u(0)dt\Big|\leq  cR^{-1/2}\sqrt{\log R}$$
where
$$B_R\widehat u(0):=\frac{1}{M_R}\int_{\D_R} \log{\frac{r}{|\zeta|}}\widehat u\omega_P
\quad\mbox{and}\quad S_t\widehat u(0):= \int_\D p(0,\cdot,t)\widehat u\omega_P.$$

A  direct  computation shows that 
$|M_R-2\pi R|$ is bounded by a constant and the area of $\D_R$ is of order $e^R$ for $R$ large. Define
$$D_R\widehat u(0):=\frac{1}{M_R}\int_{\D_R} \log{\frac{1}{|\zeta|}}\widehat u\omega_P=
\frac{2\pi}{M_R}\int_{\D_R}\Big[\int_0^\infty p(0,\zeta,t)dt\Big] \widehat u\omega_P.$$
We have
$$|B_R\widehat u(0)-D_R\widehat u(0)|\leq {1\over M_R} \log r\int_{\D_R} \omega_P
\lesssim  {1\over M_R} (1-r)e^R\lesssim {1\over R}\cdot$$
Therefore, we can replace $B_R\widehat u(0)$ with $D_R\widehat u(0)$. We have
\begin{eqnarray*}
\lefteqn{\Big|\frac{2\pi}{M_R}\int_0^{M_R\over 2\pi} S_t\widehat u(0)dt-D_R\widehat u(0)\Big|=}\\
& = & \Big |\frac{2\pi}{M_R}\int_0^{M_R\over 2\pi} \Big[\int_\D p(0,\zeta,t)\widehat u\omega_P\Big]dt-  
\frac{2\pi}{M_R}\int_0^\infty \Big[\int_{\D_R} 
p(0,\zeta,t)\widehat u\omega_P\Big]dt\Big|\\
& \leq &  \frac{2\pi}{M_R} \int_{\D_R}\Big[\int_{M_R\over 2\pi}^\infty
p(0,\zeta,t) dt \Big]\omega_P  +
\frac{2\pi}{M_R} \int_{\D\setminus \D_R} \Big[\int_0^{M_R\over 2\pi}  p(0,\zeta,t) dt\Big]\omega_P.  
\end{eqnarray*}
The two last integrals are equal since using properties of the heat kernel $p(x,y,t)$ summarized above, 
we have
$$\frac{2\pi}{M_R}\int_\D\Big[\int_0^{M_R\over 2\pi}
p(0,\zeta,t)dt\Big]\omega_P= 1 =
\frac{2\pi}{M_R}\int_{\D_R} \Big[\int_{t=0}^\infty p(0,\zeta,t)dt\Big]\omega_P.$$
Therefore,  it remains to check that
$$\frac{2\pi}{M_R}\int_{\D_R} 
\Big[\int_{M_R\over 2\pi}^\infty  p(0,\zeta,t) dt \Big] \omega_P\lesssim  R^{-1/2}\sqrt{\log R}.$$

To prove this,  split the integral over $\D_R$ into two 
integrals over $\D_R\setminus\D_{R'}$ and over $\D_{R'}$ 
where  $R':= R-R^{1/2}\sqrt{2\log R}$. 
Using the properties of the heat kernel and some direct computation, e.g. $\log|\zeta|\simeq |\zeta|-1$, 
we can bound the first part by
$$\frac{2\pi}{M_R}  \int_{\D_R\setminus \D_{R'}} \log {\frac{1}{|\zeta|}} \omega_P\lesssim
\frac{R-R'}{M_R}\lesssim \frac{R-R'}{R}\lesssim R^{-1/2}\sqrt{\log R}.$$

For the second part, we claim that
$$\int_{M_R\over 2\pi}^{\infty}  p(0,\zeta,t) dt\leq
c{R^{1/2}\over \sqrt{\log R}} \log { \frac{1}{|\zeta|}} \quad  \mbox{for}\quad
\zeta\in\D_{R'}$$
where $c>0$ is a constant independent of $R$ and $\zeta$.
Taking this for granted, since
 the Poincar{\'e} area of $\D_R$ is of order $e^R$,
using the definition of $M_R$, we get
$$\frac{1}{M_R}  
\int_{\D_R} \log {\frac{1}{|\zeta| }} \omega_P=1-{1\over M_R}\log r\int_{\D_R}
\omega_P\lesssim 1+{1\over R}(1- r) e^R\lesssim 1.$$
Therefore, our second integral to estimate is bounded by a constant times
$${1\over \sqrt{\log R}}   R^{-1/2}   \frac{1}{M_R}  
\int_{\D_R} \log {\frac{1}{|\zeta| }} \omega_P\lesssim
R^{-1/2}{1\over \sqrt{\log R}}\lesssim R^{-1/2}\sqrt{\log R}.$$
Now, it remains to prove the above claim.

Denote by $\rho$ the Poincar{\'e} distance between 0 and $\zeta$. It is given by the formula
$$\rho:=\log {1+|\zeta|\over 1-|\zeta|}\cdot$$
Recall a formula in Chavel \cite[p.246]{Chavel} 
$$p(0,\zeta,t)={\sqrt{2}e^{-t/4} \over (2\pi t)^{3/2}}
\int_{\rho}^{\infty}{s  e^{-{s^2 \over 4t}}\over \sqrt{\cosh s-\cosh\rho}}ds.$$
So, we have $\sqrt{\cosh s-\cosh\rho}\gtrsim e^{s/2}$ when $s\geq \rho+1$ and 
$\sqrt{\cosh s-\cosh\rho}\gtrsim e^{s/2}\sqrt{s-\rho}$ for $\rho<s\leq \rho+1$. 
Since $|M_R-2\pi R|$ is bounded by a constant, the integral in the claim is bounded up to some constants by
\begin{eqnarray*}
\lefteqn{\int_R^\infty\Big[{\sqrt{2}e^{-t/4} \over (2\pi t)^{3/2}}
\int_\rho^\infty {se^{-{s^2 \over 4t}} \over \sqrt{\cosh s-\cosh \rho}}ds\Big]dt}\\
&\lesssim& \int_R^\infty\Big[{\sqrt{2}e^{-t/4} \over (2\pi
  t)^{3/2}} \int_\rho^\infty {s e^{-{s^2\over 4t}} \over e^{s/2}}ds\Big] dt  +
 \int_R^\infty\Big[{\sqrt{2}e^{-t/4} \over (2\pi t)^{3/2}}
\int_\rho^{\rho+1}{s e^{-{s^2\over 4t}}\over e^{s/2} 
\sqrt{s-\rho}}ds\Big] dt\\
&\lesssim&   \int_R^\infty {1 \over \sqrt{t}}\Big[ \int_\rho^\infty
{s\over 2\sqrt{t}} e^{-\big({s\over 2\sqrt{t}} +
{\sqrt{t}\over 2}\big)^2}d\big({s\over 2\sqrt{t}}\big)\Big]dt
+\int_R^\infty\Big[{\sqrt{2}e^{-t/4}\over (2\pi
  t)^{3/2}} {(\rho+1)  e^{-{\rho^2\over 4t}} \over e^{\rho/2}}\Big]dt\\
&\lesssim&   \int_R^\infty {1\over \sqrt{t}} e^{- \big ({\rho\over 2\sqrt{t}} 
+{\sqrt{t} \over 2}\big)^2}dt \lesssim e^{-\rho}\int_R^\infty {1\over \sqrt{t}} e^{- \big
  ({\rho\over 2\sqrt{t}} -{\sqrt{t} \over 2}\big)^2}dt.
\end{eqnarray*}
We used here that $\rho\leq R\leq t$. 
Observe that the function $(t,\rho)\mapsto {\sqrt{t} \over 2}-
{\rho\over 2\sqrt{t}}$ is  increasing in $t$ and  decreasing  in
$\rho$. Therefore,
$$ {\sqrt{t} \over 2}-{\rho\over 2\sqrt{t}}\geq \frac{R-R'}{2\sqrt{R}}\gtrsim \sqrt{\log R}$$
for $\rho\leq  R'\leq R$ and $t\geq R$.   
We can bound the last expression in the above sequence of inequalities
by a constant times
$$e^{-\rho}{1\over \sqrt{\log R}}\int_R^\infty  e^{- \big
  [{\rho\over 2\sqrt{t}} -{\sqrt{t} \over 2}\big]^2} 
\Big[{\sqrt{t} \over 2}- {\rho\over 2\sqrt{t}}\Big]   d \Big[{\sqrt{t}
  \over 2}- {\rho\over 2\sqrt{t}}\Big] =  {e^{-\rho} \over \sqrt{\log R}} 
e^{-\frac{(R-R')^2}{4R}} .$$ 
Finally, we obtain the claim using that $R':= R-R^{1/2}\sqrt{2\log R}$ and
$$e^{-\rho}= {1-|\zeta|\over 1+|\zeta|}\lesssim\log {\frac{1}{|\zeta| }}$$  
for $\zeta\in\D_{R'}$. This completes the proof.
\endproof

\noindent{\it Proof of Theorem \ref{th_birkhoff}.}  Let $u$ be a function in $L^1(m_P)$. It is enough to 
show that $B_Ru(a)\to \langle m_P,u\rangle$ for $m_P$-almost every $a$.
Since this is true when $u$ is constant, we can assume   without loss of generality that  $\langle m_P,u\rangle=0$.
Fix a constant $\epsilon>0$ and define 
$E_\epsilon(u):=\big\{\widetilde Bu\geq \epsilon \big\}$.
To prove  the theorem  it suffices  to show  that  $m_P(E_\epsilon(u))=0$.

By Proposition  \ref{prop_dense_laplace_c},  $\Delta_P(\Dc^0(M))$ is  dense
in the hyperplane of functions with mean 0 in $L^1(m_P)$.   
Consequently,  for  every $\delta>0$  we  can choose
a smooth function $v$ such that $\|\Delta_P v-u\|_{L^1(m_P)}<\delta$.   
We  have  
$$E_{\epsilon}(u)\subset E_{\epsilon/2}(u-\Delta_Pv)\cup E_{\epsilon/2}(\Delta_Pv).$$
Therefore,
$$m_P\big(E_{\epsilon}(u)\big)\leq m_P\big(E_{\epsilon/2}(u-\Delta_Pv)\big) +
m_P\big(E_{\epsilon/2}(\Delta_Pv)\big).$$
We have
$$B_R(\Delta_Pv)(a)= \langle T_{a,R}, i \ddbar v\rangle = \langle i\ddbar T_{a,R},  v\rangle.$$
The last integral tends to 0 uniformly on $a$ since the mass of $\ddbar T_{a,R}$ satisfies this property.
Hence, $m_P\big(E_{\epsilon/2}(\Delta_Pv)\big)=0$.

On the other hand, by Lemma  \ref{lemma_limsup},  we  have 
\begin{eqnarray*}
m_P\big(E_{\epsilon/2}(u-\Delta_Pv)\big) &=&
m_P\big(\widetilde B(u-\Delta_P v)>\epsilon/2\big)\\
&\leq& 2c\epsilon^{-1}\|u-\Delta_Pv\|_{L^1(m_P)} \leq 2c\epsilon^{-1}\delta. 
\end{eqnarray*}
Since $\delta$  is  arbitrary,  we  deduce  from the last  estimate  that
$m_P\big(E_{\epsilon}(u)\big)=0$.
This completes the proof of the theorem.
\hfill  $\square$

\begin{remark} \rm
When $T$ is not extremal, for $u$ a function in $L^p(m_P)$, $1\leq p<\infty$, we still have the convergence of $B_Ru$ pointwise $m_P$-almost everywhere and also in $L^p(m_P)$ to a function which is constant on the leaves but not necessarily constant globally. This property can be deduced from Theorems \ref{th_birkhoff} and \ref{th_ergodic_Lp} using the decomposition of $T$ into extremal currents and that $B_R$ has norm 1 in $L^p(m_P)$.
\end{remark}

If $\phi_a:\D\rightarrow L_a$ is a universal covering map over a hyperbolic leaf $L_a$ with $\phi_a(0)=a$, 
define
$$T'_{a,R}:= (\phi_a)_* \big(\log^+ \frac{r}{|\zeta|}\big).$$
We have the following result which can be applied to foliations in $\P^k$ with linearizable isolated singularities.

\begin{theorem} Under the hypothesis of Theorem \ref{th_birkhoff}, we have the convergence 
$$\|T'_{a,R}\|^{-1} T'_{a,R}\to \|T\|^{-1} T$$
for almost every point $a\in M$ with respect to $T\wedge\omega$.
\end{theorem}
\proof
Let $\alpha$ be a smooth $(1,1)$-form on $M$. We have to show that 
$$\|T'_{a,R}\|^{-1}\langle T'_{a,R},\alpha\rangle\to \|T\|^{-1}\langle T,\alpha\rangle$$
for almost every $a$ with respect to $T\wedge\omega$ or equivalently with respect to $m_P$ 
since $T$ has no mass on the set of parabolic leaves.
This property applied to a dense countable family of $\alpha$ gives the result. The restriction of $\alpha$ to the lamination can be written as $\alpha=\varphi\omega_P$ where $\varphi$ is a $L^1(m_P)$ because $\alpha\wedge T$ is a finite measure. 
Now, we have seen in the proof of Theorem \ref{th_birkhoff} that $\langle m_{a,R},\varphi\rangle$ converges to $\langle m_P,\varphi\rangle=\langle T,\alpha\rangle$. It follows that 
$$M_{R}^{-1}\langle T'_{a,R},\alpha\rangle\to \langle T,\alpha\rangle.$$
It remains to show that 
$$M_{R}^{-1}\|T'_{a,R}\|\rightarrow \|T\|.$$
But this is a consequence of the above convergence applied to $\omega$ instead of $\alpha$. 
\endproof

\begin{remark}\rm
We have seen that $\|T'_{a,R}\|\simeq M_{R}$ for almost every $a$ with respect to $T\wedge\omega$. The estimate implies that
$$\int_{r\D} |\phi_a'(z)|^2\log{r\over |z|} idz\wedge d\overline z \simeq \log {1\over 1-r}$$
which gives a quantitative information on the behavior of $\phi_a$ near the singularities. 
When the lamination admits no positive closed invariant current, it was shown in  \cite[Th. 5.3]{FornaessSibony1} that the left hand side integral tends to infinity for every $a$. 
\end{remark}

\noindent
T.-C. Dinh, UPMC Univ Paris 06, UMR 7586, Institut de
Math{\'e}matiques de Jussieu, 4 place Jussieu, F-75005 Paris,
France.\\
{\tt  dinh@math.jussieu.fr}, {\tt http://www.math.jussieu.fr/$\sim$dinh}

\medskip

\noindent
V.-A.  Nguy{\^e}n, Math{\'e}matique-B{\^a}timent 425, UMR 8628, Universit{\'e} Paris-Sud,
91405 Orsay, France. {\tt VietAnh.Nguyen@math.u-psud.fr}

\medskip

\noindent
N. Sibony, Math{\'e}matique-B{\^a}timent 425, UMR 8628, Universit{\'e} Paris-Sud,
91405 Orsay, France. {\tt Nessim.Sibony@math.u-psud.fr}

\end{document}